\documentclass[12pt]{article}
\usepackage{graphicx}
\usepackage{latexsym,amssymb}
\usepackage{amsthm}
\usepackage{indentfirst}
\usepackage{amsmath,mathtools}
\usepackage{mathrsfs}
\usepackage{enumerate}
\usepackage{color,xcolor}
\usepackage{cite}
\usepackage{xurl}
\usepackage{hyperref}
\usepackage{microtype}
\usepackage[nameinlink,capitalise,noabbrev]{cleveref}
\usepackage[a4paper,text={160true mm,230true mm},centering,top=35true mm,head=5true mm,headsep=2.5true mm,foot=8.5true mm]{geometry}
\allowdisplaybreaks
\definecolor{linkblue}{RGB}{0,70,180}
\hypersetup{colorlinks=true,
  linkcolor=linkblue,
    citecolor=linkblue,
    urlcolor=linkblue}
\numberwithin{equation}{section}

\newtheorem{theorem}{Theorem}[section]
\newtheorem{definition}{Definition}[section]

\newtheorem{proposition}[theorem]{Proposition}
\newtheorem{lemma}[theorem]{Lemma}
\newtheorem{corollary}[theorem]{Corollary}

\newtheorem*{claim*}{Claim}

\newtheorem{remark}[theorem]{Remark}
\newtheorem{theoremalph}{Theorem}

\crefname{assumption}{Assumption}{Assumptions}
\Crefname{assumption}{Assumption}{Assumptions}
\crefname{construction}{Construction}{Constructions}
\Crefname{construction}{Construction}{Constructions}

\newcommand{\Leb}{\operatorname{Leb}}
\newcommand{\Jac}{\operatorname{Jac}}
\newcommand{\pw}{\operatorname{pw}}
\newcommand{\Basin}{\operatorname{Basin}}
\newcommand{\supp}{\operatorname{supp}}

\newcommand{\diam}{\operatorname{diam}}

\newcommand{\cP}{\mathcal P}

\newcommand{\cC}{\mathcal C}

\newcommand{\cO}{\mathcal O}

\newcommand{\N}{\mathbb N}

\newcommand{\weakstar}{\mathop{\longrightarrow}\limits^{w^*}}
\newcommand{\Diff}{\operatorname{Diff}}

\newcommand{\Gr}{\operatorname{Gr}}

\begin{document}
\begin{sloppypar}
\title{Physical measures for partially hyperbolic systems with $u$-definite subbundles}
\author{Yongluo Cao and Zeya Mi\footnote{Y. Cao is partially supported by National Key R\&D Program of China (2022YFA1005802) and NSFC 12371194. Z. Mi is partially supported by the National Key R\&D Program of China (2022YFA1007800) and NSFC 12671223.}}
\date{\today}
\maketitle

\begin{abstract}
In this paper, we study \(C^{1+\alpha}\) diffeomorphisms of a compact Riemannian
manifold \(\mathbf{M}\) with an attracting set \(\Lambda\) admitting
a partially hyperbolic splitting
\[
T_\Lambda \mathbf{M}
=
E^u\oplus C_1\oplus\cdots\oplus C_k\oplus E^s.
\]
We assume that, for every ergodic Gibbs \(u\)-state supported on
\(\Lambda\), the Lyapunov exponents along each \(C_i\) are either
all positive or all negative. We prove that \(\Lambda\) supports
finitely many ergodic hyperbolic SRB measures whose basins cover
Lebesgue almost every point of its topological basin.

Our approach relies on an empirical basin-covering theorem for
attracting sets with dominated splittings. Its proof hinges on a new
notion of sliding suffix times, which allows us to synchronize
expansion and contraction along a common empirical subsequence.
\end{abstract}

\section{Introduction}

A central problem in smooth dynamics is to characterize the asymptotic
statistics of Lebesgue-typical orbits. This motivates the study of invariant probability measures governing time averages on sets of positive volume.  Let \(f\colon \mathbf{M}\to \mathbf{M}\) be a diffeomorphism of a compact
Riemannian manifold.  For an invariant probability measure \(\mu\),
its basin is
\[
\Basin(\mu)
=
\left\{
x\in \mathbf{M}:
\frac1n\sum_{j=0}^{n-1}\delta_{f^jx}
\weakstar\mu
\right\}.
\]
The measure \(\mu\) is called \emph{physical} if its basin has positive
Lebesgue measure.

A closely related geometric notion is that of an SRB measure.  An
invariant probability measure is an \emph{SRB measure} if it has a positive Lyapunov exponent almost
everywhere and its conditional measures along local Pesin unstable
manifolds are absolutely continuous with respect to leaf volume.  For
\(C^{1+\alpha}\) diffeomorphisms, every ergodic hyperbolic SRB measure is
physical; see \cite{Pesin,LedrappierYoung}.  For uniformly hyperbolic
attractors, the classical works of Sinai, Ruelle, and Bowen established
finiteness of ergodic SRB measures and basin covering; see
\cite{Sinai72,Ruelle76,Bowen75,BowenRuelle75}.

Beyond uniform hyperbolicity, one would still like to know whether typical
orbits are described by finitely many physical measures.  This belongs to
the broader program proposed by Palis on the finitude of attractors and
the prevalence of their basins \cite{Palis00}.  Partially hyperbolic
systems, and more generally systems with dominated splittings, form a
natural setting for this problem.

For partially hyperbolic diffeomorphisms, Gibbs \(u\)-states are the
natural candidates for physical measures.  Their conditional measures
along strong unstable leaves are absolutely continuous with respect to
leaf volume \cite{PesinSinai}.  Much of the theory is governed by the
signs of their center Lyapunov exponents.  When all Gibbs \(u\)-states
have negative center Lyapunov exponents, finiteness and basin covering were proved
by Bonatti--Viana \cite{BonattiViana}.  The mostly expanding theory was
developed by Alves--Bonatti--Viana and later refined in several directions
\cite{ABV,AnderssonVasquez,MCY2017,BurguetYang}.

The problem is subtler when the sign depends on the Gibbs \(u\)-state.
For a one-dimensional center, hyperbolicity already determines a sign for
each ergodic Gibbs \(u\)-state; see \cite{HuaYangYang,CYZ}.  Several
dominated one-dimensional center subbundles were considered in
\cite{MiCao2026,BurguetCenter}.  For one center subbundle of arbitrary
dimension, Cao--Mi--Yang \cite{CaoMiYang} introduced the
\(u\)-definite condition: for every ergodic Gibbs \(u\)-state, all
exponents along the center have one strict sign, although the sign may
vary with the measure.  They proved finiteness and basin covering under
this assumption.

The sign condition rules out two different obstructions.  Zero center Lyapunov
exponents cannot in general be allowed: the product of an Anosov
diffeomorphism with the identity on the circle has ergodic Gibbs
\(u\)-states with vanishing center Lyapunov exponent but no physical measure.  In higher
dimension, hyperbolicity alone is also insufficient.  The Bowen-eye
products in \cite{CaoMiYang} have hyperbolic ergodic Gibbs \(u\)-states,
with exponents of both signs along the center, but no physical measure.

We study physical and SRB measures for partially hyperbolic diffeomorphisms with several \(u\)-definite center subbundles.

\medskip
Let $f$ be a
\(C^{1+\alpha}\) diffeomorphism on \(\mathbf{M}\).  A compact set \(\Lambda\subset \mathbf{M}\) is said to be an attracting set of $f$ if there exists a trapping
region \(U\) such that
\(
f(\overline U)\subset U,
\Lambda=\bigcap_{n\ge0}f^n(U).
\)
Its topological basin is defined as $B(\Lambda):=\{x\in \mathbf{M}:\omega(x)\subset\Lambda\}$.
The set \(\Lambda\) is \emph{partially hyperbolic} if there exists an invariant dominated splitting
$$
T_\Lambda \mathbf{M}
=
E^u\oplus C_1\oplus\cdots\oplus C_k\oplus E^s
$$
of the tangent space over $\Lambda$, such that $E^u$ is uniformly expanding and $E^s$ is uniformly contracting, see \S \ref{subsec:dominated-splittings} for more details.
We write \(E^c=C_1\oplus\cdots\oplus C_k\).

\begin{definition}

An invariant subbundle \(C\subset E^c\) is \emph{\(u\)-definite} if,
for every ergodic Gibbs \(u\)-state \(\mu\) supported on \(\Lambda\),
the Lyapunov exponents along \(C\) are either all positive or all
negative.
\end{definition}

The sign may depend on both \(i\) and the ergodic Gibbs \(u\)-state, and no
dominated refinement inside an individual \(C_i\) is assumed.  Our main
result is the following.

\begin{theoremalph}\label{thm:A}
Let $f$ be a $C^{1+\alpha}$ diffeomorphism on a compact Riemannian manifold $\mathbf{M}$ and $\Lambda$ be an attracting set with partially hyperbolic splitting
$$
T_\Lambda \mathbf{M}
=
E^u\oplus C_1\oplus\cdots\oplus C_k\oplus E^s.
$$
If every center subbundle \(C_i\) is
\(u\)-definite, then \(\Lambda\) supports only finitely many ergodic
hyperbolic SRB measures \(\mu_1,\ldots,\mu_s\).  These measures are
physical and satisfy
\[
\Leb\left(
B(\Lambda)\setminus\bigcup_{i=1}^s\Basin(\mu_i)
\right)=0.
\]
\end{theoremalph}

Theorem \ref{thm:A} recovers the one-center result of Cao--Mi--Yang \cite{CaoMiYang} when
$k=1$. If every $C_i$ is one-dimensional, $u$-definiteness is equivalent
to the absence of zero center Lyapunov exponents for ergodic Gibbs
$u$-states. In this case, Theorem \ref{thm:A} removes the additional empirical-sign
assumption used for basin covering in \cite{MiCao2026}.

Burguet \cite{BurguetCenter} proved finiteness and basin covering for dominated
one-dimensional center subbundles under hyperbolicity of the ergodic
SRB measures. Our assumption is instead imposed on ergodic Gibbs
$u$-states, but allows center subbundles of arbitrary dimension without
requiring a dominated refinement. Moreover, different center subbundles may
have opposite signs for the same ergodic Gibbs $u$-state, so the full
center need not be $u$-definite. Thus neither the one-dimensional
arguments of \cite{BurguetCenter,MiCao2026} nor a successive application of the one-center
result of \cite{CaoMiYang} applies directly.

We next establish a general basin-covering criterion for attracting
sets with a dominated splitting. It is the main ingredient in
Theorem~\ref{thm:A}, but does not require $u$-definiteness. For $x\in B(\Lambda)$,
let $\pw(x)$ denote the set of weak-* accumulation points of its
empirical measures. We write $\lambda_E^-$ and $\lambda_F^+$ for
the minimal and maximal Lyapunov exponents along the indicated
bundles.

\begin{theoremalph}\label{thm:B}
Let $f$ be a $C^{1+\alpha}$ diffeomorphism on a compact Riemannian manifold $\mathbf{M}$ and let \(\Lambda\) be an attracting
set with a dominated splitting
\(
T_\Lambda \mathbf{M}=E\oplus F.
\)
Assume that there is a measurable subset \(\Gamma\) of $B(\Lambda)$ such that for every \(x\in\Gamma\), there exists
\(\mu\in\pw(x)\) such that
\[
\lambda_E^-(y)>0
\quad\text{for \(\mu\)-almost every }y,
\qquad
\mu\{y:\lambda_F^+(y)<0\}>0.
\]
Then, for Lebesgue almost every \(x\in\Gamma\), there exists an
ergodic hyperbolic SRB measure \(\nu\), supported on \(\Lambda\), such
that
\[
 x\in\Basin(\nu),
 \qquad
 \lambda_E^-(\nu)>0>\lambda_F^+(\nu).
\]
\end{theoremalph}

The hypotheses are asymmetric: with respect to the same empirical
accumulation point, expansion along $E$ is required almost everywhere,
whereas contraction along $F$ is required only on a set of positive
measure. This asymmetry matches the center-sign patterns imposed by
domination and $u$-definiteness, allowing the criterion to yield
existence and basin covering in Theorem~\ref{thm:A}.

The key ingredient is a notion of \emph{sliding suffix time}, which
controls expansion along sufficiently long suffixes in
higher-dimensional bundles. Under the expansion hypothesis, the
density of these times along the given empirical subsequence can
be made arbitrarily close to one. This yields a positive density
of synchronized expansion--contraction times along the same
subsequence. Using these synchronized times and the Gibbs estimates associated
with sliding suffix times, we establish a F{\o}lner--Gibbs property
following the variational approach of Burguet--Yang~\cite{BurguetYang}. Combined with a dynamical density argument, this gives the hyperbolic SRB measures and basin covering asserted in
Theorem~\ref{thm:B}.

In Section~\ref{sec:examples}, we construct a connected attracting set
with two higher-dimensional \(u\)-definite center subbundles whose sum
is not \(u\)-definite. Neither subbundle admits a nontrivial dominated
refinement.

\medskip
The paper is organized as follows.
Section~\ref{sec:preliminaries} recalls the preliminaries on
dominated splittings, Lyapunov exponents, entropy, and Gibbs states.
Section~\ref{sec:proof-A} deduces Theorem~\ref{thm:A} from
Theorem~\ref{thm:B} through basin covering and a finiteness argument.
Section~\ref{sec:temporal} establishes temporal synchronization
of expansion along \(E\) and contraction along \(F\).
Section~\ref{sec:geometry} develops the geometry of sliding suffix
times and proves dynamical density and multi-component Gibbs estimates.
Section~\ref{sec:marked-basin} combines F{\o}lner selection with
these estimates to prove a leafwise basin theorem, from which
Theorem~\ref{thm:B} follows.
Finally, Section~\ref{sec:examples} constructs examples and
establishes the robustness of their \(u\)-definiteness.

\paragraph{Use of AI tools.} AI were used for improving the presentation of the text. They also played a crucial role in constructing examples provided in Section~\ref{sec:examples}. The authors take full responsibility for the mathematical content and the
final text.

\section{Preliminaries}\label{sec:preliminaries}

Let $\mathbf{M}$ be a compact Riemannian manifold without boundary, and let $f \in \Diff^1(\mathbf{M})$. We denote by $\Leb$ the normalized Riemannian volume on $\mathbf{M}$, and by $\Leb_D$ the induced Riemannian volume on a smooth submanifold $D \subset \mathbf{M}$. For a linear isomorphism $A$ between finite-dimensional normed vector spaces, its conorm is defined by
\[
  m(A) := \inf_{v \neq 0} \frac{\|Av\|}{\|v\|} = \|A^{-1}\|^{-1}.
\]
For each $x \in \mathbf{M}$ and $n \ge 1$, we define the empirical measure along the orbit of $x$ as
\[
  \mu_n^x := \frac{1}{n}\sum_{j=0}^{n-1} \delta_{f^j x},
\]
where $\delta_y$ stands for the Dirac mass at $y$. Let $\pw(x)$ be the set of weak-$\ast$ accumulation points of the sequence $\{\mu_n^x\}_{n \ge 1}$ in the space of Borel probability measures. By standard arguments (see, e.g., \cite[Theorem~6.9]{Walter}), every measure in $\pw(x)$ is $f$-invariant.

We write \(\Jac(A)\) for the absolute Riemannian Jacobian of an
injective linear map \(A\). For an interval \(I\subset\mathbb R\),
write \(I_{\mathbb Z}=I\cap\mathbb Z\). Intervals indexing iterates
are understood as integer intervals when the subscript is omitted.

\subsection{Dominated splittings and partial hyperbolicity}
\label{subsec:dominated-splittings}

Let $\Lambda \subset \mathbf{M}$ be a compact $f$-invariant set. A continuous $Df$-invariant splitting $T_\Lambda \mathbf{M} = E \oplus_{\succ} F$ is said to be \emph{dominated} if there exist constants $C > 0$ and $\tau \in (0,1)$ such that
\[
  \bigl\|Df_x^n|_{F(x)}\bigr\| \le C \tau^n m\bigl(Df_x^n|_{E(x)}\bigr)
  \qquad \text{for all } x \in \Lambda \text{ and } n \ge 1.
\]
An invariant subbundle $E^u \subset T_\Lambda \mathbf{M}$ is \emph{uniformly expanding}, and an invariant subbundle $E^s \subset T_\Lambda \mathbf{M}$ is \emph{uniformly contracting}, if there exist constants $C > 0$ and $\tau \in (0,1)$ such that
\[
  m\bigl(Df_x^n|_{E^u(x)}\bigr) \ge C^{-1} \tau^{-n}
  \quad \text{and} \quad
  \bigl\|Df_x^n|_{E^s(x)}\bigr\| \le C \tau^n
  \qquad \text{for all } x \in \Lambda \text{ and } n \ge 1,
\]
respectively. The set $\Lambda$ is called \emph{partially hyperbolic} if it admits a multi-bundle dominated splitting
\[
  T_\Lambda \mathbf{M} = E^u \oplus_{\succ} C_1 \oplus_{\succ} \cdots \oplus_{\succ} C_k \oplus_{\succ} E^s,
\]
where $E^u$ is uniformly expanding and $E^s$ is uniformly contracting.

\subsection{Lyapunov exponents and metric entropy}

By Oseledets' Multiplicative Ergodic Theorem \cite{O68}, for any $f$-invariant probability measure $\mu$, there exists a Borel set $\mathcal{R} \subset \mathbf{M}$ of full $\mu$-measure such that for every $x \in \mathcal{R}$, the Lyapunov exponents
\[
  \lambda_1(x) \ge \lambda_2(x) \ge \cdots \ge \lambda_{\dim \mathbf{M}}(x)
\]
are well-defined (counted with multiplicity). An invariant measure $\mu$ is called \emph{hyperbolic} if $\lambda_i(x) \neq 0$ for all $1 \le i \le \dim \mathbf{M}$ and $\mu$-almost every $x$.

For a continuous $Df$-invariant bundle $G \subset T_\Lambda \mathbf{M}$, the maximal and minimal Lyapunov exponents along $G$ at a regular point $x \in \Lambda$ are given by
\[
  \lambda_G^+(x) := \lim_{n\to\infty} \frac{1}{n} \log \bigl\|Df_x^n|_{G(x)}\bigr\|,
  \qquad
  \lambda_G^-(x) := \lim_{n\to\infty} \frac{1}{n} \log m\bigl(Df_x^n|_{G(x)}\bigr).
\]
For an invariant probability measure $\mu$ supported on $\Lambda$, we define the integrated maximal and minimal exponents as
\[
  \lambda_G^\pm(\mu) := \int_\Lambda \lambda_G^\pm(x) \, d\mu(x).
\]
If $\mu$ is ergodic, then $\lambda_G^\pm(x) = \lambda_G^\pm(\mu)$ for $\mu$-almost every $x$. By Kingman's Subadditive Ergodic Theorem \cite{Kingman}, for fixed $f$ and $G$, the map $\mu \mapsto \lambda_G^+(\mu)$ is upper semicontinuous and the map $\mu \mapsto \lambda_G^-(\mu)$ is lower semicontinuous with respect to the weak-$\ast$ topology on the space of invariant probability measures supported on $\Lambda$.

\medskip
Let $\mathcal{P}$ be a finite Borel partition of $\mathbf{M}$. For any finite index set $Q \subset \mathbb{Z}_{\ge 0}$, we write $\mathcal{P}^Q := \bigvee_{t\in Q} f^{-t}\mathcal{P}$, and abbreviate $\mathcal{P}^r := \bigvee_{t=0}^{r-1} f^{-t}\mathcal{P}$ for $r \ge 1$. For a Borel probability measure $m$, the measure-theoretic entropy of $\mathcal{P}$ is
$$
  H_m(\mathcal{P}) := -\sum_{P \in \mathcal{P}} m(P) \log m(P) = -\int_{\mathbf{M}} \log m(\mathcal{P}(x)) \, dm(x),
$$
where $\mathcal{P}(x)$ denotes the partition atom containing $x$, with the convention $0 \log 0 = 0$. For an $f$-invariant probability measure $\mu$, its metric entropy with respect to $f$ is defined as
$$
  h_\mu(f) := \sup_{\mathcal{P}} h_\mu(f, \mathcal{P}),
  \qquad \text{where} \quad
  h_\mu(f, \mathcal{P}) := \lim_{r\to\infty} \frac{1}{r} H_\mu(\mathcal{P}^r),
$$
and the supremum is taken over all finite Borel partitions of $\mathbf{M}$.

\subsection{Gibbs $u$-states and Gibbs $E$-states}

Throughout this subsection, we assume that $f \in \Diff^{1+\alpha}(\mathbf{M})$ for some $\alpha > 0$. Let $\Lambda \subset \mathbf{M}$ be a partially hyperbolic attracting set. We denote by $\mathcal{G}^u(f)$ the collection of $f$-invariant probability measures supported on $\Lambda$ whose conditional measures along strong unstable manifolds are absolutely continuous with respect to the induced Lebesgue measure on these manifolds. Elements of $\mathcal{G}^u(f)$ are called \emph{Gibbs $u$-states}. We collect several fundamental properties of Gibbs $u$-states below; see \cite{PesinSinai, CYZ, CaoMiYang, MiCao2026}.

\begin{proposition}\label{prop:gibbs-u-basic}
  Let $f \in \Diff^{1+\alpha}(\mathbf{M})$ and let $\Lambda$ be a partially hyperbolic attracting set.
  \begin{enumerate}[\rm (i)]
    \item Almost every ergodic component of a Gibbs $u$-state is a Gibbs $u$-state.
    \item The set $\mathcal{G}^u(f)$ of Gibbs $u$-states is compact in the weak-$\ast$ topology.
    \item There exists a set $\mathcal{R}_u \subset B(\Lambda)$ of full Lebesgue measure in the basin of attraction $B(\Lambda)$ such that
    \[
      \pw(x) \subset \mathcal{G}^u(f) \qquad \text{for every } x \in \mathcal{R}_u.
    \]
  \end{enumerate}
\end{proposition}

More generally, let $\Lambda$ be a compact invariant set admitting a dominated splitting $T_\Lambda \mathbf{M} = E \oplus_{\succ} F$. At any regular point $x \in \Lambda$ satisfying $\lambda_E^-(x) > 0$, Pesin theory guarantees the existence of a \emph{Pesin unstable manifold} $W^E_{\mathrm{loc}}(x)$ tangent to $E(x)$ with $\dim W^E_{\mathrm{loc}}(x) = \dim E$.

\begin{definition}[Gibbs $E$-state]
  An $f$-invariant probability measure $\mu$ supported on $\Lambda$ is called a \emph{Gibbs $E$-state} (or a Gibbs $cu$-state associated with $E$) if:
  \begin{enumerate}[\rm (1)]
    \item $\lambda_E^-(x) > 0$ for $\mu$-almost every $x \in \Lambda$;
    \item The conditional measures of $\mu$ along local Pesin unstable manifolds $\{W^E_{\mathrm{loc}}(x)\}$ are absolutely continuous with respect to the induced Lebesgue measure on these manifolds.
  \end{enumerate}
\end{definition}

In particular, when $E$ is uniformly expanding, a Gibbs $E$-state reduces to a standard Gibbs $u$-state. The following properties of Gibbs $E$-states are derived from \cite[Propositions~2.4 and~2.6, Theorem~3.3]{MiCao2026}.

\begin{proposition}\label{prop:gibbs-E-basic}
  Let $T_\Lambda \mathbf{M} = E \oplus_{\succ} F$ be a dominated splitting.
  \begin{enumerate}[\rm (i)]
    \item Almost every ergodic component of a Gibbs $E$-state is a Gibbs $E$-state.
    \item Let $\{\mu_n\}_{n \ge 1}$ be a sequence of ergodic Gibbs $E$-states converging in the weak-$\ast$ topology to an invariant measure $\mu$. If there exists a constant $\chi > 0$ such that $\lambda_E^-(x) \ge \chi$ for $\mu$-almost every $x$, then $\mu$ is also a Gibbs $E$-state.
  \end{enumerate}
\end{proposition}

\begin{proposition}\label{thm:gibbs-E-accumulation}
  Let $T_\Lambda \mathbf{M} = E \oplus_{\succ} F$ be a dominated splitting. Suppose that a sequence of pairwise distinct ergodic Gibbs $E$-states $\{\mu_n\}_{n \ge 1}$ converges weak-$\ast$ to a Gibbs $E$-state $\mu$. Then
  \[
    \lambda_F^+(x) \ge 0 \qquad \text{for } \mu\text{-almost every } x.
  \]
\end{proposition}

We also recall the nesting property for hierarchical Gibbs states \cite[Lemma~3.1]{MiCao2026}.

\begin{lemma}\label{lem:nested-gibbs-states}
  Let $T_\Lambda \mathbf{M} = G_0 \oplus_{\succ} G_1 \oplus_{\succ} \cdots \oplus_{\succ} G_r \oplus_{\succ} F$ be a dominated splitting. Set $E_j := \bigoplus_{i=0}^j G_i$ for each $0 \le j \le r$. Then for any $0 \le i < j \le r$, every Gibbs $E_j$-state is also a Gibbs $E_i$-state.
\end{lemma}

\section{Deduction of Theorem~\ref{thm:A} from Theorem~\ref{thm:B}}
\label{sec:proof-A}

In this section we deduce Theorem~\ref{thm:A} from
Theorem~\ref{thm:B}, whose proof is deferred to
Section~\ref{sec:marked-basin}. We first cover a full Lebesgue measure subset of the topological basin
by finitely many sets, according to the Lyapunov-sign patterns of
empirical limits. Applying Theorem~\ref{thm:B} to these subsets yields basin covering by ergodic hyperbolic SRB measures. We then prove that only
finitely many such measures occur.

Let $\Lambda$ be an attracting set with partially hyperbolic splitting
$$
T_\Lambda \mathbf{M}
=
E^u\oplus_{\succ} C_1\oplus_{\succ}\cdots\oplus_{\succ} C_k\oplus_{\succ} E^s.
$$
For \(0\le i\le k\), set
$$
 E_i
 =
 E^u\oplus_{\succ} C_1\oplus_{\succ}\cdots\oplus_{\succ} C_i,
 \qquad
 F_i
 =
 C_{i+1}\oplus_{\succ}\cdots\oplus_{\succ} C_k\oplus_{\succ} E^s,
$$
where \(E_0=E^u\) and \(F_k=E^s\). By the domination of the original
splitting, each \(E_i\oplus F_i\) is a dominated splitting.

\subsection{Empirical index selection and basin covering}

\begin{proposition}
\label{prop:basin-covering-flag}
Under the assumptions of Theorem~\ref{thm:A}, Lebesgue almost every
point of \(B(\Lambda)\) belongs to the basin of an ergodic hyperbolic SRB
measure supported on \(\Lambda\).
\end{proposition}

\begin{proof}
Let \(\mathcal R_u\subset B(\Lambda)\) be the full Lebesgue measure set provided by Proposition~\ref{prop:gibbs-u-basic}(iii), so that $\pw(x)\subset\mathcal G^u(f)$ for every $x\in\mathcal R_u$.
For
\(0\le i\le k\), define
\[
 \Gamma_i
 :=
 \left\{
 x\in\mathcal R_u:
 \begin{array}{l}
 \text{there exists \(\mu\in\pw(x)\) such that}\\[-1mm]
 \lambda_{E_i}^-(y)>0\ \text{for \(\mu\)-a.e. }y,
 \quad
 \mu\{y:\lambda_{F_i}^+(y)<0\}>0
 \end{array}
 \right\}.
\]
We claim that
\begin{equation}\label{eq:empirical-index-cover}
 \mathcal R_u=\bigcup_{i=0}^k\Gamma_i.
\end{equation}
Since every \(\Gamma_i\) is contained in \(\mathcal R_u\), it suffices
to prove the reverse inclusion. Fix \(x\in\mathcal R_u\), choose
\(\mu\in\pw(x)\), and write its ergodic decomposition as
\[
 \mu=\int\nu\,d\tau_\mu(\nu).
\]
The measure \(\mu\) is a Gibbs \(u\)-state, so
\(\tau_\mu\)-almost every \(\nu\) is an ergodic Gibbs \(u\)-state.
By domination and \(u\)-definiteness, for almost every such \(\nu\)
there is a unique index
\(
 \iota(\nu)\in\{0,\ldots,k\}
\)
such that
\[
 \lambda_{E_{\iota(\nu)}}^-(\nu)>0>
 \lambda_{F_{\iota(\nu)}}^+(\nu).
\]
Indeed, the center blocks with positive exponents form an initial
segment of \(C_1,\ldots,C_k\), while the remaining blocks have
negative exponents.

Let \(i\) be the smallest index such that
\[
 \tau_\mu\{\nu:\iota(\nu)=i\}>0.
\]
Then \(\iota(\nu)\ge i\) for \(\tau_\mu\)-almost every \(\nu\).
Consequently,
\[
 \lambda_{E_i}^-(y)>0
 \qquad\text{for \(\mu\)-almost every }y,
\]
whereas
\[
 \mu\{y:\lambda_{F_i}^+(y)<0\}
 \ge
 \tau_\mu\{\nu:\iota(\nu)=i\}
 >0.
\]
Thus \(x\in\Gamma_i\), proving
\eqref{eq:empirical-index-cover}.

For every \(0\le i\le k\), apply Theorem~\ref{thm:B} to
\(\Gamma_i\) and the dominated splitting \(E_i\oplus F_i\). It follows
that Lebesgue almost every point of \(\Gamma_i\) belongs to the basin
of an ergodic hyperbolic SRB measure supported on \(\Lambda\).
Since there are only finitely many indices \(i\),
\eqref{eq:empirical-index-cover} and the full-measure property of
\(\mathcal R_u\) give the conclusion.
\end{proof}

\subsection{Finiteness and completion of the proof}

Proposition~\ref{prop:basin-covering-flag} reduces
Theorem~\ref{thm:A} to the following finiteness statement.

\begin{proposition}\label{prop:finiteness}
Under the assumptions of Theorem~\ref{thm:A}, there are only finitely many
ergodic hyperbolic SRB measures supported on \(\Lambda\).
\end{proposition}

The proof relies on the following accumulation property of Gibbs
\(E_i\)-states.

\begin{lemma}
\label{lem:prefix-gibbs-accumulation}
Under the assumptions of Theorem~\ref{thm:A}, fix \(0\le i\le k\).  Let \(\{\mu_n\}_{n\ge1}\) be a sequence of
pairwise distinct ergodic Gibbs \(E_i\)-states such that
\(
\mu_n\weakstar\mu.
\)
Then \(\mu\) is a Gibbs \(E_i\)-state.  Moreover, there exists
\(\chi>0\) such that
\[
\lambda_{E_i}^-(x)\ge\chi
\qquad\text{for \(\mu\)-a.e. }x.
\]
\end{lemma}
\begin{proof}
We argue by induction on \(i\).  The case \(i=0\) follows from the
weak-* compactness of Gibbs \(u\)-states and the uniform expansion of
\(E_0=E^u\).

Assume the result for \(i-1\). By
Lemma~\ref{lem:nested-gibbs-states}, the measures \(\mu_n\) are
ergodic Gibbs \(E_{i-1}\)-states.  Hence the induction hypothesis
implies that \(\mu\) is a Gibbs \(E_{i-1}\)-state.
Applying Proposition~\ref{thm:gibbs-E-accumulation} to
\(E_{i-1}\oplus F_{i-1}\), we obtain
\[
\lambda_{F_{i-1}}^+(x)\ge0
\qquad\text{for \(\mu\)-a.e. }x.
\]
Since \(F_{i-1}=C_i\oplus F_i\), domination and \(u\)-definiteness give
\[
\lambda_{C_i}^-(x)>0
\qquad\text{for \(\mu\)-a.e. }x.
\]

We claim that this positivity is uniform.  Otherwise there exist
pairwise distinct ergodic components \(\nu_m\) of \(\mu\) such that
\[
\lambda_{C_i}^-(\nu_m)\longrightarrow0.
\]
Up to considering a subsequence, we assume \(\nu_m\weakstar\nu\).
Since \(\mu\) is a Gibbs \(E_{i-1}\)-state, it follows from
Proposition~\ref{prop:gibbs-E-basic}(i) that
\(\nu_m\) are ergodic Gibbs \(E_{i-1}\)-states.
The induction hypothesis implies that \(\nu\) is a Gibbs
\(E_{i-1}\)-state.  Applying
Proposition~\ref{thm:gibbs-E-accumulation} to \(\{\nu_m\}\), and arguing
again with \(u\)-definiteness, yields
\(
\lambda_{C_i}^-(x)>0
\)
for $\nu$-a.e. $x$.
Hence
\begin{equation}\label{cip}
\lambda_{C_i}^-(\nu)
=
\int\lambda_{C_i}^-(x)\,d\nu(x)
>0.
\end{equation}
On the other hand, Kingman's formula gives
\[
\lambda_{C_i}^-(\eta)=\sup_{n\ge 1}\frac{1}{n}\int \log m(Df_x^n|_{C_i(x)})d\eta(x)
\]
for every $f$-invariant measure $\eta$; see \cite{Kingman}.  Thus
$\eta\mapsto \lambda_{C_i}^-(\eta)$ is lower semicontinuous, and
\[
\lambda_{C_i}^-(\nu)
\le
\liminf_{m\to\infty}\lambda_{C_i}^-(\nu_m)
=0,
\]
which is a contradiction to \eqref{cip}.
Thus the Lyapunov exponents of \(\mu\) along \(C_i\) are uniformly
positive.  By domination, the same is true along
$
E_i=E_{i-1}\oplus C_i.
$
Hence, there exists \(\chi>0\) such that
\[
\lambda_{E_i}^-(x)\ge\chi
\qquad\text{for \(\mu\)-a.e. }x.
\]
Proposition~\ref{prop:gibbs-E-basic}(ii) therefore implies that
\(\mu\) is a Gibbs \(E_i\)-state. The induction is
complete.
\end{proof}

\begin{proof}[Proof of Proposition~\ref{prop:finiteness}]
For \(0\le i\le k\), let \(\mathscr P_i\) be the family of ergodic
hyperbolic SRB measures \(\mu\), supported on \(\Lambda\), such that
$$
\lambda_{E_i}^-(\mu)>0>\lambda_{F_i}^+(\mu).
$$
Every measure in \(\mathscr P_i\) is an ergodic Gibbs \(E_i\)-state.
Conversely, \(u\)-definiteness and domination place every ergodic
hyperbolic SRB measure in exactly one \(\mathscr P_i\).  It remains to
prove that these families are finite.

Suppose first that \(i<k\) and that \(\mathscr P_i\) is infinite.
Choose pairwise distinct
\[
\mu_n\in\mathscr P_i,
\qquad
\mu_n\weakstar\mu.
\]
By Lemma~\ref{lem:prefix-gibbs-accumulation}, \(\mu\) is a Gibbs
\(E_i\)-state.  Applying Proposition~\ref{thm:gibbs-E-accumulation} to
\(E_i\oplus F_i\), and then using that \(C_{i+1}\) is the strongest
subbundle in \(F_i\) together with \(u\)-definiteness, gives
\[
\lambda_{C_{i+1}}^-(x)>0
\qquad\text{for \(\mu\)-a.e. }x.
\]
Consequently,
\begin{equation}\label{eq:Ci1-positive-measure}
\lambda_{C_{i+1}}^-(\mu)>0.
\end{equation}
For every \(n\), however,
\(\lambda_{C_{i+1}}^-(\mu_n)<0\).
By the lower semicontinuity of the minimal exponent functional,
\[
\lambda_{C_{i+1}}^-(\mu)
\le
\liminf_{n\to\infty}\lambda_{C_{i+1}}^-(\mu_n)
\le0,
\]
contradicting \eqref{eq:Ci1-positive-measure}.  Thus
\(\mathscr P_i\) is finite for \(i<k\).

Finally, suppose that \(\mathscr P_k\) is infinite.  Taking a weak-*
limit of pairwise distinct measures in \(\mathscr P_k\),
Lemma~\ref{lem:prefix-gibbs-accumulation} gives a Gibbs \(E_k\)-state
\(\mu\).  Proposition~\ref{thm:gibbs-E-accumulation}, applied to
\(E_k\oplus E^s\), produces a nonnegative Lyapunov exponent along
\(E^s\), contradicting its uniform contraction.  Thus
\(\mathscr P_k\) is finite as well.
\end{proof}

We now complete the proof of Theorem~\ref{thm:A} by applying Propositions~\ref{prop:basin-covering-flag} and
\ref{prop:finiteness}.

\begin{proof}[Proof of Theorem~\ref{thm:A}]
By Proposition~\ref{prop:basin-covering-flag}, Lebesgue almost every
point of \(B(\Lambda)\) belongs to the basin of an ergodic hyperbolic SRB
measure supported on \(\Lambda\). Proposition~\ref{prop:finiteness}
shows that there are only finitely many such measures, which completes the proof of Theorem~\ref{thm:A}.
\end{proof}

\section{Temporal synchronization}\label{sec:temporal}

A key first step in the proof of Theorem~\ref{thm:B} is to establish
a temporal synchronization mechanism. Roughly speaking, if an
empirical subsequence converges to a measure satisfying the required
Lyapunov-sign conditions, then along the same subsequence there is a
positive density of times at which expansion along \(E\) and
contraction along \(F\) occur simultaneously.

We first introduce sliding suffix times and their expansion estimates
along \(E\). We then define the contraction carriers along \(F\) and
prove the synchronization theorem.

Throughout this section and the next two sections, let $f\in \Diff^{1}(\mathbf{M})$ and $\Lambda$ be an attracting set with a dominated splitting $T_{\Lambda}\mathbf{M}=E\oplus_{\succ} F.$
Fix a sufficiently small trapping region \(U\), with
\(f(\overline U)\subset U\), and a continuous extension
\(\widetilde E\oplus\widetilde F\) of the splitting to
\(\overline U\). These extensions need not be \(Df\)-invariant.
Fixing \(\theta>0\) sufficiently small, we consider the $E$-direction cone field $\cC_\theta^E=(\cC_{\theta}^E(x))_{x\in \overline{U}}$ of width $\theta$ by
$$
\cC_\theta^E(x)
=
\left\{
v=v^E+v^F:
v^E\in\widetilde E(x),\
v^F\in\widetilde F(x),
\|v^F\|\le\theta\|v^E\|
\right\}.
$$
By domination, after choosing an adapted metric and shrinking \(U\),
there exists \(0<\tau_0<1\) such that
\begin{equation}\label{invc}
 Df_x(\cC_\theta^E(x))
 \subset\cC_{\tau_0\theta}^E(fx)
 \qquad \text{for every } x\in\overline U.
\end{equation}

\subsection{Sliding suffix times and cone expansion}

To detect the local expansion behavior along the $E$-direction for points starting from the trapping region, we
define
$$
 A_n^E(x)
 :=
 \log
 \inf_{\substack{
 v\in\mathcal C_{\theta}^E(x)\\
 \|v\|=1}}
 \|Df_x^nv\|,\qquad x\in\overline U, ~n\ge 0.
$$
By construction, each \(A_n^E\) is continuous, and there exists \(C>0\) such that
\[
 |A_n^E(x)|\le Cn,
 \qquad x\in\overline U, ~n\ge 0.
\]
Forward invariance of the cone field gives the superadditivity of $\{A_n^E\}_{n\ge 0}$:
$$
 A_{n+m}^E(x)
 \ge
 A_n^E(x)+A_m^E(f^nx),
 \qquad n, ~m\ge 0.
$$

Now we can introduce the notion of sliding suffix times.

\begin{definition}\label{def:sliding-suffix}
Fix \(c>0\) and integers \(N\ge L\ge1\). For $x\in U$, an integer \(n\ge N\) is
called a \((c,L,N)\)-\emph{sliding suffix time} for \(E\) at \(x\) if
$$
 \frac{1}{r-L+1}
 \sum_{s=n-r}^{n-L}\frac1L A_L^E(f^s x)
 \ge c
 \qquad
 \text{for every }N\le r\le n.
$$
The set of all such times is denoted by
\(
 S_{c,L,N}^E(x).
\)
\end{definition}

By definition, $n \in S_{c,L,N}^E(x)$ means that every suffix ending at $n$ of length at least $N$ has average $L$-step $E$-expansion at least $c$, taken over all sliding length-$L$ windows contained in it (Figure~\ref{fig:sliding-suffix}); individual windows need not be expanding. The parameters $L$ and $N$ serve distinct purposes: the block length $L$ captures expansion governed by the minimal Lyapunov exponent along $E$—which one-step conorms may underestimate in higher dimensions—while $N$ absorbs boundary losses from incomplete blocks, yielding robust expansion on long suffixes (Lemma~\ref{lem:robust-suffix-conorm}). Ultimately, as will be shown below, this mechanism provides sliding suffix times of density arbitrarily close to one, enabling the synchronization of $E$-expansion with positive-density contraction along $F$.

\begin{figure}[htbp]
 \centering
 \includegraphics[width=.85\textwidth]{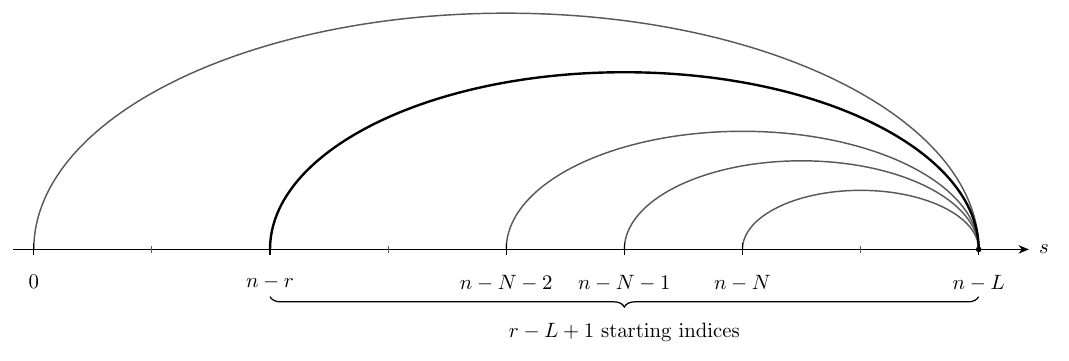}
 \caption{The arcs represent the intervals of window starting
indices \([n-r,n-L]_{\mathbb Z}\), with \(N\le r\le n\).}
 \label{fig:sliding-suffix}
\end{figure}

\begin{remark}
We use cone expansions because forward invariance of the cone
field ensures superadditivity on the trapping region \(U\),
where \(\widetilde E\) need not be \(Df\)-invariant.
When \(\Lambda=\mathbf{M}\), one may replace \(A_L^E(x)\) in
Definition~\ref{def:sliding-suffix} with
$\log m\bigl(Df_x^L|_{E(x)}\bigr)$.
Since the logarithmic conorms are superadditive and differ
from the cone expansions by a uniformly bounded amount
(Lemma~\ref{lem:cone-bundle-comparison}), all subsequent arguments carry through up to adjusting constants.
\end{remark}

The following comparison identifies the asymptotic cone expansion
on \(\Lambda\) with the minimal Lyapunov exponent along \(E\).

\begin{lemma}
\label{lem:cone-bundle-comparison}
There exists \(K_E\ge1\) such that, for every \(y\in\Lambda\) and
every \(n\ge1\),
$$
\log m\bigl(Df_y^n|_{E(y)}\bigr)-\log K_E
\le
A_n^E(y)
\le
\log m\bigl(Df_y^n|_{E(y)}\bigr).
$$
Consequently, for every invariant probability
measure $\mu$ supported on $\Lambda$,
\begin{equation}\label{eq:cone-exponent-identification}
\lim_{n\to\infty}\frac1nA_n^E(y)
=
\lambda_E^-(y),
\qquad\text{for \(\mu\)-a.e. }y.
\end{equation}
\end{lemma}

\begin{proof}
The upper bound follows from \(E(y)\subset\cC_{\theta}^E(y)\).  For any \(v=v^E+v^F\in\cC_{\theta}^E(y)\) with $y\in \Lambda$, we have 
\[
(1-\theta)\|v^E\|\le\|v\|\le(1+\theta)\|v^E\|.
\]
If \(\|v\|=1\), invariance of the splitting and of the cone field gives
\[
\|Df_y^nv\|
\ge(1-\theta)\|Df_y^nv^E\|
\ge\frac{1-\theta}{1+\theta}
   m(Df_y^n|_{E(y)}).
\]
Taking the infimum over unit cone vectors proves the lower bound with
\(K_E=(1+\theta)/(1-\theta)\).  Dividing by \(n\) and letting
\(n\to\infty\) yields \eqref{eq:cone-exponent-identification} by Oseledets'
theorem.
\end{proof}

The next result shows that sliding suffix times yield a robust finite-time expansion estimate.
\begin{lemma}
\label{lem:robust-suffix-conorm}
Fix $0<\sigma<c, L\in \N$. There exist an integer \(N_0\ge L\) and \(\delta_{L}\in (0,1)\) such
that the following holds for every \(N\ge N_0\).
For every \(x\in U\),
\(b\in S_{c,L,N}^E(x)\) and \(N\le r\le b\), if \(z\in U\) satisfies
\[
 d(f^jz,f^{b-r+j}x)<\delta_{L}
 \qquad\text{for every }0\le j\le r,
\]
then
\begin{equation}\label{eq:robust-suffix-conorm}
 A_r^E(z)\ge\sigma r.
\end{equation}
In particular, for every \(x\in U\), \(N\ge N_0\), and
\(b\in S_{c,L,N}^E(x)\),
\begin{equation}\label{eq:robust-reference-cone-score}
 A_r^E(f^{b-r}x)\ge\sigma r \qquad\text{for every }N\le r\le b.
\end{equation}
\end{lemma}

\begin{proof}
Let $R_L:=\max_{0\le h<L}\sup_{y\in\overline U}(-A_h^E(y))$ and fix \(\varepsilon\in (0, c-\sigma)\). Choose  \(N_0\in \N\) so that
\[
 N_0\ge\max\left\{L,\frac{c(L-1)+2R_L}{(c-\varepsilon-\sigma)}\right\}.
\]
By uniform continuity of
\(A_L^E\) on \(\overline U\), there is \(\delta_{L}\in (0,1)\) such that
$$
 |A_L^E(y)-A_L^E(z)|<\varepsilon L
 \quad\text{if }y,z\in\overline U,\ d(y,z)<\delta_{L}.
$$

Now we fix \(N\ge N_0\), \(x\in U\), \(b\in S_{c,L,N}^E(x)\), and
\(N\le r\le b\).
For \(0\le k<L\), put
\[
 q_k:=\left\lfloor\frac{r-k}{L}\right\rfloor,
 \qquad
 \mathscr R_k:=\{b-r+k+jL:0\le j<q_k\}.
\]
A class is empty when \(q_k=0\). A direct calculation shows that
\[
 [b-r,b-L]_{\mathbb Z}
 =\bigsqcup_{k=0}^{L-1}\mathscr R_k.
\]
Since $b\in S_{c,L,N}^E(x)$, one can choose \(k\in\{0,\ldots,L-1\}\) such that
\begin{equation}\label{eq:selected-residue-score}
 \sum_{s\in\mathscr R_k}A_L^E(f^sx)
=\sum_{j=0}^{q_k-1}A_L^E(f^{b-r+k+jL}x)
 \ge c(r-L+1).
\end{equation}
Set
$q:=q_k$ and $t:=r-k-qL$.
Then we have $q\ge 1$, $0\le k,t<L$ and $r=k+qL+t$.
Suppose that $z\in U$ and \(
 d(f^jz,f^{b-r+j}x)<\delta_{L}\) for every $0\le j \le r$.
Superadditivity of $\{A_n^E\}$, \eqref{eq:selected-residue-score} and the choice of \(\delta_L\) then yield
\[
 \begin{aligned}
 A_r^E(z)
 &\ge A_k^E(z)
      +\sum_{j=0}^{q-1}A_L^E(f^{k+jL}z)
      +A_t^E(f^{k+qL}z)\\
 &\ge \sum_{j=0}^{q-1}\left(A_L^E(f^{b-r+k+jL}x)
      -\varepsilon L \right)-2R_L\\
 &\ge c(r-L+1)-\varepsilon qL-2R_L\\
 &\ge(c-\varepsilon)r-c(L-1)-2R_L,
 \end{aligned}
\]
where the last inequality uses \(qL\le r\). Since \(r\ge N\ge N_0\), by the choice of $N_0$ we obtain from above that \(A_r^E(z)\ge\sigma r\), proving
\eqref{eq:robust-suffix-conorm}. As \(r\in[N,b]_{\mathbb Z}\) was arbitrary, taking
\(z=f^{b-r}x\) for each such \(r\) yields
\eqref{eq:robust-reference-cone-score}.
\end{proof}

\subsection{Synchronization along an empirical subsequence}

\paragraph{Mildly hyperbolic carriers.}

For integers \(\ell,M\ge1\) and \(\eta>0\), set
\[
 \psi_\ell^F(x)
 :=-\frac1\ell
 \log\bigl\|Df_x^\ell|_{\widetilde F(x)}\bigr\|,
 \qquad x\in\overline U.
\]
For \(x\in U\), let \(\mathcal I_x\) be the union of the intervals
\([\max\{0,p-M\},p)_{\mathbb Z}\) over all integers \(p\ge1\) satisfying
\[
 \sum_{j=p-k}^{p-1}\psi_\ell^F(f^jx)\ge\eta k
 \qquad \text{whenever } 1\le k\le\min\{M,p\}.
\]
The \emph{mildly hyperbolic carrier along \(F\)},
\(G_{\eta,\ell,M}^{F}(x)\)\footnote{In the notation of Burguet~\cite{BurguetCenter},
\(G_{\eta,\ell,M}^{F}(x)=G_a^{\eta,M}((1))\), where
\(a=(\psi_\ell^F(f^nx))_{n\ge0}\).}, consists of those integers \(t\ge0\)
for which there exists a connected component \(I\) of \(\mathcal I_x\) such that
\begin{itemize}
\item $t, t+1 \in I$,
\item for every $q\in I$ with $q\le t$,
$$
 \sum_{j=q}^{t}\psi_\ell^F(f^jx)
 \ge\frac\eta2(t+1-q).
$$
\end{itemize}

The following theorem synchronizes these carriers with sliding suffix
times along the same empirical subsequence.

\begin{theorem}
\label{thm:empirical-synchronization}
Let \(x\in U\) and \(n_j\to\infty\) satisfy  \(
 \mu_{n_j}^x\weakstar\mu
\). Suppose that
$$
 \lambda_E^-(y)>0
 \quad\text{for $\mu$-almost every }y,
 \qquad
 \mu\{y:\lambda_F^+(y)<0\}>0.
$$
Then there exist \(c,\eta,\rho>0\) and integers \(L,\ell,M_0\ge1\)
such that, for every \(N\ge L\) and \(M\ge M_0\),
$$
 \liminf_{j\to\infty}\frac1{n_j}
 \#\left(
 S_{c,L,N}^{E}(x)
 \cap G_{\eta,\ell,M}^{F}(x)
 \cap[0,n_j)
 \right)
 \ge\rho.
$$
Moreover, \(c,\eta,\rho\) may be chosen rational.
\end{theorem}

We deduce the theorem from the following two density estimates.

\begin{proposition}
\label{lem:sliding-suffix-saturation}
Let \(x\in U\) and \(n_j\to\infty\) satisfy  \(
 \mu_{n_j}^x\weakstar\mu
\). Suppose that
\[
 \lambda_E^-(y)>0
 \qquad\text{for \(\mu\)-almost every }y.
\]
Then, for every \(\varepsilon>0\), there exist \(c\in \mathbb Q_{>0}\) and \(L\ge1\)
such that, for every fixed \(N\ge L\),
$$
 \limsup_{j\to\infty}\frac1{n_j}
 \#\left([0,n_j)\setminus S_{c,L,N}^{E}(x)\right)
 <\varepsilon.
$$
\end{proposition}

\begin{proposition}
\label{lem:subsequence-mild-carrier}
Let \(x\in U\) and \(n_j\to\infty\) satisfy  \(
 \mu_{n_j}^x\weakstar\mu
\).
Suppose that
\(
 \mu\{y:\lambda_F^+(y)<0\}>0.
\)
Then there exist \(\eta\in \mathbb Q_{>0}, \gamma>0\), \(\ell\ge1\), and \(M_0\ge1\)
such that, for every fixed \(M\ge M_0\),
$$
 \liminf_{j\to\infty}\frac1{n_j}
 \#\bigl(G_{\eta,\ell,M}^{F}(x)\cap[0,n_j)\bigr)
 \ge\gamma.
$$
\end{proposition}

Assuming Propositions~\ref{lem:sliding-suffix-saturation}
and~\ref{lem:subsequence-mild-carrier}, we first prove
Theorem~\ref{thm:empirical-synchronization}.

\begin{proof}[Proof of Theorem~\ref{thm:empirical-synchronization}]
Let \(\eta,\gamma,\ell,M_0\) be given by
Proposition~\ref{lem:subsequence-mild-carrier}.  Apply
Proposition~\ref{lem:sliding-suffix-saturation} with
\(\varepsilon=\gamma/2\), obtaining \(c\in \mathbb Q_{>0}\) and \(L\ge1\).  Fix \(N\ge L\) and \(M\ge M_0\). For every \(j\), we have
\[
\begin{aligned}
&\#\Bigl(
S_{c,L,N}^{E}(x)
 \cap G_{\eta,\ell,M}^{F}(x)
 \cap[0,n_j)
 \Bigr)\\
&\qquad\ge
 \#\Bigl(
 G_{\eta,\ell,M}^{F}(x)\cap[0,n_j)
 \Bigr)
 -
 \#\Bigl(
 [0,n_j)\setminus S_{c,L,N}^{E}(x)
 \Bigr).
\end{aligned}
\]
After dividing by \(n_j\), Proposition~\ref{lem:subsequence-mild-carrier}
bounds the lower limit of the first term on the right by \(\gamma\),
whereas Proposition~\ref{lem:sliding-suffix-saturation} bounds the
upper limit of the second by a number smaller than
\(\gamma/2\). As a consequence,
\[
\liminf_{j\to\infty}\frac1{n_j}
\#\Bigl(
 S_{c,L,N}^{E}(x)
 \cap G_{\eta,\ell,M}^{F}(x)
 \cap[0,n_j)
\Bigr)
>
\frac{\gamma}{2}.
\]
Choosing any rational \(0<\rho<\gamma/2\) completes the proof.
\end{proof}

We now turn to the proofs of the two density estimates.

\subsubsection{The \(E\)-estimate}

Let us record the weak-\((1,1)\) estimate for the one-sided discrete
Hardy--Littlewood maximal operator
(see e.g. \cite[Chapter~2]{Duoandikoetxea}).  We provide a proof for
completeness.

\begin{lemma}
\label{lem:one-sided-discrete-maximal}
Let \(u_0,\ldots,u_{n-1}\ge0\), and define
\[
 \mathscr M^-u(m)
 :=
 \max_{0\le q\le m}
 \frac1{m-q+1}\sum_{s=q}^{m}u_s,
 \qquad 0\le m<n.
\]
Then, for every \(a>0\),
$$
 \#\{0\le m<n:\mathscr M^-u(m)>a\}
 \le
 \frac1a\sum_{s=0}^{n-1}u_s.
$$
\end{lemma}

\begin{proof}
Let
\(
 \mathcal E=\{0\le m<n:\mathscr M^-u(m)>a\}.
\)
If \(\mathcal E=\varnothing\), there is nothing to prove.  Otherwise,
set \(\mathcal E_0=\mathcal E\).  Inductively, as long as
\(\mathcal E_{j-1}\ne\varnothing\), define
\(
 r_j:=\max\mathcal E_{j-1}.
\)
Since \(r_j\in\mathcal E\), the definition of \(\mathscr M^-u(r_j)\)
provides an integer \(0\le\ell_j\le r_j\) such that
\begin{equation}\label{eq:maximal-witness-interval}
 \sum_{s=\ell_j}^{r_j}u_s>a(r_j-\ell_j+1).
\end{equation}
Set
\[
 I_j:=[\ell_j,r_j]_{\mathbb Z},
 \qquad
 \mathcal E_j:=\mathcal E_{j-1}\setminus I_j.
\]
Because \(r_j\in\mathcal E_{j-1}\cap I_j\), the cardinality of the
remaining set strictly decreases, so the construction stops after
finitely many steps, say with \(\mathcal E_J=\varnothing\).
By construction,
\begin{itemize}
\item the intervals $I_j$, $1\le j \le J$ are pairwise disjoint,
\item
\(
 \mathcal E\subset\bigcup_{j=1}^J I_j.
\)
\end{itemize}
Together with \eqref{eq:maximal-witness-interval}, we obtain
\[
 a\,\#\mathcal E
 \le a\sum_{j=1}^J|I_j|
 <\sum_{j=1}^J\sum_{s\in I_j}u_s
 \le\sum_{s=0}^{n-1}u_s.
\]
The proof is complete.
\end{proof}

Using this tool, we can provide the proof of Proposition~\ref{lem:sliding-suffix-saturation} as follows:

\begin{proof}[Proof of Proposition~\ref{lem:sliding-suffix-saturation}]
Write \(b_L=L^{-1}A_L^E\). By Lemma~\ref{lem:cone-bundle-comparison},
\(b_L\to\lambda_E^-\) almost everywhere; since \(\{b_L\}\) is uniformly
bounded, dominated convergence gives \(b_L\to\lambda_E^-\) in
\(L^1(\mu)\). Moreover, \(\lambda_E^->0\) a.e. implies
\[
  \frac1c\int(2c-\lambda_E^-)_+\,d\mu
  \le 2\mu\{0<\lambda_E^-<2c\}\to0 \qquad(c\to0).
\]
Thus, for \(\varepsilon>0\), one can choose \(c\in \mathbb Q_{>0}\) and
\(L\in\mathbb N\) such that
\[
  \frac1c\int\bigl((2c-\lambda_E^-)_++|b_L-\lambda_E^-|\bigr)\,d\mu
  <\varepsilon.
\]
Put \(d_L=(2c-b_L)_+\). With the fact \(d_L\le(2c-\lambda_E^-)_++|b_L-\lambda_E^-|\), the above estimate yields
\begin{equation}\label{eq:block-defect-integral}
  \frac1c\int d_L\,d\mu<\varepsilon.
\end{equation}

Fix \(N\ge L\). If \(n\ge N\) is not in
\(S_{c,L,N}^{E}(x)\), then there exists \(N\le r\le n\) such that
\[
  \sum_{s=n-r}^{n-L}b_L(f^sx)<c(r-L+1).
\]
Since we have \(b_L\ge2c-d_L\) by definition, this gives
\[
  \sum_{s=m-K+1}^{m}d_L(f^sx)>cK,
\]
where \(K=r-L+1\) and \(m=n-L\).
Thus \(\mathscr M^-u(m)>c\) for \(u_s=d_L(f^sx)\). Bad indices
\(n\ge N\) correspond to distinct \(m=n-L\), so
Lemma~\ref{lem:one-sided-discrete-maximal} gives
\[
  \#\bigl([0,n_j)\setminus S_{c,L,N}^{E}(x)\bigr)
  \le N+\frac1c\sum_{s=0}^{n_j-1}d_L(f^sx)
\]
for all sufficiently large $j$.
Since \(d_L\) is continuous and \(\mu_{n_j}^x\weakstar\mu\),
\[
  \lim_{j\to \infty}\frac1{n_j}\sum_{s=0}^{n_j-1}d_L(f^sx)
  =\int d_L\,d\mu.
\]
Dividing by \(n_j\), taking \(\limsup_{j\to\infty}\), and using
\eqref{eq:block-defect-integral}, we get
\[
  \limsup_{j\to\infty}
  \frac1{n_j}\#\bigl([0,n_j)\setminus S_{c,L,N}^{E}(x)\bigr)
  \le \frac1c\int d_L\,d\mu
  <\varepsilon.
\]
This completes the proof.
\end{proof}

\subsubsection{The \(F\)-estimate}

Recall that \(\mathcal I_x\) is the interval union used to define
\(G_{\eta,\ell,M}^{F}(x)\).

\begin{lemma}
\label{lem:finite-window-mildification}
For every \(\eta>0\), there exists
\(\vartheta=\vartheta(\eta)\in(0,1)\) such that, for every
\(x\in U\) and all integers \(\ell,M,n\ge1\),
\[
 \#\bigl(G_{\eta,\ell,M}^{F}(x)\cap[0,n)\bigr)
 \ge
 \vartheta\,\#\bigl(\mathcal I_x\cap[0,n)\bigr)
 -\left(2+\frac nM+M\right).
\]
\end{lemma}

\begin{proof}
Fix \(\eta>0\).  Choose \(K>\eta\) such that
\(
 |\psi_\ell^F(w)|\le K
 \)
for every $w\in\overline U$ and every $\ell\ge1$.
Fix \(x\in U\) and integers \(\ell,M,n\ge1\).
Let \(\mathcal C_n\) be the family of connected components of
\(\mathcal I_x\cap[0,n)\).
Except possibly for the components containing \(0\) or \(n-1\),
each component contains at least \(M\) integers. Hence
\begin{equation}\label{eq:ncomponents}
 \#\mathcal C_n\le 2+
 \frac{
 \#\bigl(\mathcal I_x\cap[0,n)\bigr)
 }{M}\le2+\frac nM.
\end{equation}

Fix \(I=[u,v]_{\mathbb Z}\in\mathcal C_n\).
Starting from \(t_0=u\), whenever \(t_h\le v\), choose an interval
in the definition of \(\mathcal I_x\) containing \(t_h\), and
denote its right endpoint by \(t_{h+1}\).
Then
\[
 t_h<t_{h+1}\le t_h+M,
 \qquad
 \sum_{k=t_h}^{t_{h+1}-1}\psi_\ell^F(f^kx)
 \ge\eta(t_{h+1}-t_h).
\]
Stop at the first index \(s\) with \(t_s>v\).
Summing gives
\begin{equation*}
 \sum_{k=u}^{t_s-1}\psi_\ell^F(f^kx)
 \ge\eta(t_s-u).
\end{equation*}
Moreover, \([u,t_s)_{\mathbb Z}\) lies in the full component of
\(\mathcal I_x\) whose left endpoint is \(u\).
If \(v<n-1\), then \(t_s=v+1\); otherwise \(v+1\) would belong to
\(\mathcal I_x\cap[0,n)\), contradicting the choice of \(I\).
If \(v=n-1\), then
\[
 n\le t_s\le n+M-1.
\]

By Pliss' lemma \cite[Lemma~3.1]{ABV}, there exist
\(\vartheta=\vartheta(\eta,K)\in(0,1)\) and
\(\mathcal P_I\subset[u+1,t_s]_{\mathbb Z}\) such that
\(
 \#\mathcal P_I\ge\vartheta(t_s-u)
\)
and for each $p\in\mathcal P_I$,
\begin{equation}\label{eq:component-pliss-endpoints}
 \sum_{k=q}^{p-1}\psi_\ell^F(f^kx)
 \ge\frac\eta2(p-q)
 \qquad \text{for every }\ u\le q<p.
\end{equation}
For every \(p\in\mathcal P_I\setminus\{t_s\}\), both \(p-1\)
and \(p\) lie in the same component of \(\mathcal I_x\).
Since \(u\) is its left endpoint,
\eqref{eq:component-pliss-endpoints} gives
\(
 p-1\in G_{\eta,\ell,M}^{F}(x).
\)
Consequently, if \(v<n-1\), then
$$
 \#\bigl(G_{\eta,\ell,M}^{F}(x)\cap I\bigr)
 \ge\vartheta\,\#I-1.
$$
For the possible component \(I=[u,n-1]_{\mathbb Z}\),
discarding the times in \([n,t_s)_{\mathbb Z}\) gives
$$
 \begin{aligned}
 \#\bigl(G_{\eta,\ell,M}^{F}(x)\cap I\bigr)
 &\ge\vartheta(t_s-u)-1-(t_s-n)\\
 &\ge\vartheta\,\#I-M.
 \end{aligned}
$$
Summing over \(\mathcal C_n\) and using
\eqref{eq:ncomponents}, we obtain
\[
 \begin{aligned}
 \#\bigl(G_{\eta,\ell,M}^{F}(x)\cap[0,n)\bigr)
 &\ge
 \vartheta\,\#\bigl(\mathcal I_x\cap[0,n)\bigr)
 -\#\mathcal C_n-M\\
 &\ge
 \vartheta\,\#\bigl(\mathcal I_x\cap[0,n)\bigr)
 -\left(2+\frac nM+M\right).
 \end{aligned}
\]
\end{proof}

The next lemma turns a positive Birkhoff tail into weak carriers along the
same empirical subsequence.

\begin{lemma}
\label{lem:quantitative-weak-time}
Let \(\phi\in C(\mathbf{M})\), and set
\[
 \phi^*(y)
 :=
 \limsup_{q\to\infty}\frac1q
 \sum_{s=0}^{q-1}\phi(f^sy),
 \qquad
 \Phi_x
 :=
 \bigl(\phi(f^nx)\bigr)_{n\ge0}.
\]
For \(\delta>0\) and an integer \(M\ge1\), let
\(W_{\delta,M}^{\Phi_x}\) be the union of the integer intervals
\([\max\{0,p-M\},p)\) over all integers \(p\ge1\) satisfying
\[
 \sum_{q=p-r}^{p-1}\phi(f^qx)\ge\delta r
 \qquad \text{for every } 1\le r\le\min\{M,p\}.
\]
Suppose that \(\mu_{n_j}^x\weakstar\mu\) and that, for some
\(a,m_0>0\),
\[
 \mu\bigl(\{y:\phi^*(y)>a\}\bigr)\ge m_0.
\]
Then, for every \(0<\delta<a\), there exists \(M_0\ge1\) such that,
for every fixed integer \(M\ge M_0\),
\[
 \liminf_{j\to\infty}\frac1{n_j}
 \#\bigl(
 W_{\delta,M}^{\Phi_x}\cap[0,n_j)
 \bigr)
 \ge\frac{m_0}{2}.
\]
\end{lemma}

\begin{proof}
Choose
\(
 \delta<\delta'<a,
\)
and write the ergodic decomposition of \(\mu\) as
\(
 \mu=\int\nu\,d\tau_\mu(\nu).
\)
By Birkhoff's ergodic theorem, for \(\tau_\mu\)-almost every ergodic component \(\nu\), we have
\(
 \phi^*(y)=\int\phi\,d\nu
 \)
 for \(\nu\)-almost every $y$.
Thus, by taking
$
\Omega_a=\left\{\nu: \int\phi\,d\nu>a\right\},
$
we conclude that
\begin{equation}\label{eq:positive-ergodic-components}
 \tau_\mu(\Omega_a)
 =
 \int
 \nu\{y:\phi^*(y)>a\}\,d\tau_\mu(\nu) =
 \mu\{y:\phi^*(y)>a\}
 \ge m_0.
\end{equation}

Let
\[
 A_{\delta'}
 =
 \left\{
 z:
 \sum_{q=0}^{m-1}
 \bigl(\delta'-\phi(f^{-q}z)\bigr)>0
 \text{ for some }m\ge1
 \right\}.
\]
For any \(\nu\in \Omega_a\), by applying the maximal ergodic theorem for \(f^{-1}\) to the function
\(\delta'-\phi\), we get
\(
 \int_{A_{\delta'}}(\delta'-\phi)\,d\nu\ge0.
\)
Note also that
\(
 \int(\delta'-\phi)\,d\nu
 =
 \delta'-\int\phi\,d\nu
 <0.
\)
It follows that
\(
 \int_{A_{\delta'}^c}(\delta'-\phi)\,d\nu<0,
\)
and hence
\(
 \nu(A_{\delta'}^c)>0.
\)
Define
\[
 H_\delta=
 \left\{z:\sum_{q=1}^{m}\phi(f^{-q}z)>m\delta
 \text{ for every }m\ge1\right\}.
\]
By definition, we see $A_{\delta'}^c\subset f^{-1}H_\delta$ and hence
$
\nu(H_\delta)=\nu(f^{-1}H_\delta)\ge \nu(A_{\delta'}^c)>0.
$
The ergodicity of $\nu$ gives
$\nu(K_\delta)=1$, where $K_\delta=\bigcup_{k\ge1}f^{-k}H_\delta$. Using arbitrariness of $\nu\in \Omega_a$ and \eqref{eq:positive-ergodic-components}, we obtain
\begin{equation}\label{eq:Km}
 \mu(K_\delta)\ge \int_{\Omega_a}\nu(K_{\delta})d\tau_{\mu}(\nu)=\tau_{\mu}(\Omega_a)\ge m_0.
\end{equation}

For every $M\ge 1$, consider the two open sets
\[
 H_{\delta,M}=
 \left\{z:\sum_{q=1}^{m}\phi(f^{-q}z)>m\delta
 \text{ for every }1\le m\le M\right\},\quad\quad  K_{\delta,M}=\bigcup_{k=1}^{M}f^{-k}H_{\delta,M}.
\]
Note that $\bigcup_{k=1}^{M}f^{-k}H_\delta$ increases to $K_{\delta}$ as $M\to \infty$. This together with \eqref{eq:Km} and the fact $H_\delta\subset H_{\delta,M}$ for every $M\ge 1$, yields $M_0\ge 1$ such that
$$
\mu(K_{\delta,M})\ge \mu\left(\bigcup_{k=1}^{M}f^{-k}H_\delta\right)\ge m_0/2\qquad \textrm{for every } M\ge M_0.
$$
Fix \(M\ge M_0\). Since $K_{\delta,M}$ is
open and $\mu_{n_j}^x\weakstar\mu$, it follows that
\begin{equation}\label{eq:OM-portmanteau}
 \liminf_{j\to\infty}\mu_{n_j}^x(K_{\delta,M})\ge \mu(K_{\delta,M})\ge \frac{m_0}{2}.
\end{equation}
We next relate \(K_{\delta,M}\) to the weak carrier $W_{\delta,M}^{\Phi_x}$ and show that
\begin{equation}\label{kfin}
f^{\ell}(x)\in K_{\delta,M}\Longrightarrow \ell \in W_{\delta,M}^{\Phi_x}.
\end{equation}
Indeed, suppose that $f^\ell(x)\in K_{\delta,M}$. Then for some $1\le k\le M$ the point $f^{\ell+k}(x)$ belongs to $H_{\delta,M}$.  Hence, for every
$1\le m\le\min\{M,\ell+k\}$,
\[
 \sum_{q=\ell+k-m}^{\ell+k-1}\phi(f^q x)
 =\sum_{r=1}^{m}\phi(f^{-r}f^{\ell+k} (x))>m\delta.
\]
Thus \([\max\{0,\ell+k-M\},\ell+k)\) is one of the intervals
defining \(W_{\delta,M}^{\Phi_x}\). Since \(1\le k\le M\), it
contains \(\ell\), proving \eqref{kfin}.
Consequently, for every \(j\),
$$
 \frac1{n_j}
 \#\bigl(W_{\delta,M}^{\Phi_x}\cap[0,n_j)\bigr)
 \ge
 \mu_{n_j}^x(K_{\delta,M}).
$$
Taking the lower limit and using
\eqref{eq:OM-portmanteau} proves the result.
\end{proof}

We now apply these estimates to the \(F\)-direction.

\begin{proof}[Proof of Proposition~\ref{lem:subsequence-mild-carrier}]
Choose \(a\in\mathbb Q_{>0}\) such that
\[
 B:=\{y:\lambda_F^+(y)\le-2a\},
 \qquad m:=\mu(B)>0.
\]
Write \(\mu=\int\nu\,d\tau_\mu(\nu)\). Since \(B\) is invariant, we have either $\nu(B)=1$ or $\nu(B)=0$ for $\tau_{\mu}$-almost every $\nu$. Thus,
$$
\tau_{\mu}(\Omega_B)=\int \nu(B)d\tau_{\mu}(\nu)=\mu(B)=m,
$$
where $\Omega_B=\{\nu: \nu(B)=1\}=\{\nu: \lambda_F^+(\nu)\le -2a\}$. 
For each \(\nu\in\Omega_B\), Kingman's theorem gives
\[
 \lim_{\ell\to\infty}\int\psi_\ell^F\,d\nu
 =-\lambda_F^+(\nu)\ge2a.
\]
By Egorov's theorem, we can choose $\ell\ge 1$ so that the set 
$
\Omega'=\{\nu\in \Omega_B: \int \psi_\ell^Fd\nu>a\}
$
satisfies $\tau_{\mu}(\Omega')\ge m/2$.
By Birkhoff's ergodic theorem,
\[
 \mu\{y:(\psi_\ell^F)^*(y)>a\}
 \ge\int_{\Omega'}\nu\{y:(\psi_\ell^F)^*(y)>a\}
       \,d\tau_\mu(\nu)
 =\tau_\mu(\Omega')\ge\frac m2.
\]

Extend \(\psi_\ell^F\) continuously to \(\mathbf{M}\) and choose
\(\eta\in\mathbb Q\cap(0,a)\).
For \(\phi=\psi_\ell^F\), the set \(W_{\eta,M}^{\Phi_x}\)
in Lemma~\ref{lem:quantitative-weak-time} is precisely
\(\mathcal I_x\). Applying that lemma with
\(\delta=\eta\) and \(m_0=m/2\), we obtain \(M_1\ge1\)
such that, for every fixed \(M\ge M_1\),
\begin{equation}\label{eq:F-weak-carrier-density}
 \liminf_{j\to\infty}\frac1{n_j}
 \#\bigl(\mathcal I_x\cap[0,n_j)\bigr)\ge\frac m4.
\end{equation}
Lemma~\ref{lem:finite-window-mildification} then gives
\(\vartheta>0\) such that
\begin{equation}\label{gpr}
 \#\bigl(G_{\eta,\ell,M}^{F}(x)\cap[0,n_j)\bigr)
 \ge
 \vartheta\,\#\bigl(\mathcal I_x\cap[0,n_j)\bigr)
 -\left(2+\frac{n_j}{M}+M\right).
\end{equation}
Choose an integer $M_0\ge \max\{M_1, 8/(\vartheta m)\}$ and set $\gamma:=\vartheta m/8$.
Dividing \eqref{gpr} by \(n_j\) and using
\eqref{eq:F-weak-carrier-density}, we obtain, for every fixed
\(M\ge M_0\),
\[
 \liminf_{j\to\infty}\frac1{n_j}
 \#\bigl(G_{\eta,\ell,M}^{F}(x)\cap[0,n_j)\bigr)
 \ge\frac{\vartheta m}{4}-\frac1M\ge\gamma.
\]
\end{proof}

\section{Geometry of sliding suffix times}\label{sec:geometry}

This section develops the geometry associated with sliding suffix
times. To this end, we first construct suffix preballs and prove bounded distortion and dynamical density on disks tangent to the \(E\)-cone field. We then establish a Gibbs volume estimate on
unions of time intervals.

Throughout this section and the next, we assume in addition that
\(f\in\Diff^{1+\alpha}(\mathbf{M})\) for some \(\alpha>0\).
We say that a $C^1$ submanifold $D\subset U$ is tangent to the $E$-direction cone field $\cC_{\theta}^E$ if $\dim D=\dim E$, and $T_xD\subset \cC_{\theta}^E(x)$ for each point $x\in D$. Then all its iterates are tangent to $\cC_{\theta}^E$ as well, by the invariance property \eqref{invc}.

For a submanifold \(D\), let \(d_D\) denote its intrinsic distance and
\(B_D(x,r)=\{y\in D:d_D(x,y)<r\}\). Distances to \(\partial D\)
are measured intrinsically.

\subsection{Suffix preball geometry and dynamical density}

\subsubsection{Suffix preballs and bounded distortion}

Fix a sufficiently small exponential-chart radius \(\delta_0>0\)
such that \(\exp_x^{-1}\) is defined on \(B(x,\delta_0)\)
for every \(x\in \mathbf{M}\).

\begin{proposition}\label{prop:suffix-preball}
Fix \(0<\sigma<c\) and an integer \(L\ge1\). There exists an integer \(N_0=N_0(\sigma,c,L,\theta)\ge L\) such that for every \(N\ge N_0\), there exist constants \(K_N\ge1\) and \(\rho_N \in (0,\delta_0)\) with the following properties.
Let \(D\subset U\) be a \(C^1\) embedded disk tangent to \(\mathcal{C}_\theta^E\).

\begin{enumerate}[(i)]
    \item\label{bcp} For every \(x\in D\) and every \(b\in S_{c,L,N}^E(x)\), we have
    \[
    d_{f^{b-j}(D)}\bigl(f^{b-j}(x),f^{b-j}(y)\bigr)
    \le K_N e^{-\sigma j}\, d_{f^b(D)}\bigl(f^b(x),f^b(y)\bigr),\qquad 0\le j\le b,
    \]
    for any point \(y\in D\) satisfying
    \(
    d_{f^b(D)}\bigl(f^b(x),f^b(y)\bigr)\le \rho_N.
    \)
    \item\label{tls} For every \(r_*>0\), there exists an integer \(b_*=b_*(r_*,N)\ge N\) such that for every \(x\in D\) with \(d_D(x,\partial D)\ge r_*\) and every \(b\in S_{c,L,N}^E(x)\) with \(b\ge b_*\), we have
    \[
    d_{f^b(D)}\bigl(f^b x,\partial f^b(D)\bigr)>\rho_N.
    \]
\end{enumerate}
\end{proposition}

\begin{proof}
Take \(N_0\) and \(\delta_L\in (0,1)\) from Lemma~\ref{lem:robust-suffix-conorm}. Fix \(N\ge N_0\), and choose \(K_N\ge1\) and \(\rho_N\in(0,\delta_0)\) such that
\begin{equation}\label{fkv}
    K_N\ge \max_{0\le j\le N}\bigl\{e^{\sigma j}\|Df^{-j}\|_\infty\bigr\},
    \qquad
    2K_N\rho_N<\delta_L.
\end{equation}

We first prove \eqref{bcp}. Fix \(x\in D\) and \(b\in S_{c,L,N}^E(x)\). Given \(y\in D\) with \(d_{f^b(D)}(f^b(x),f^b(y))\le \rho_N\), let \(\gamma\) be a length-minimizing \(C^1\) curve in \(f^b(D)\) connecting \(f^b(x)\) and \(f^b(y)\). For \(0\le r\le b\), put \(\gamma_r=f^{-r}\gamma\). It suffices to show
\begin{equation}\label{eq:terminalpullback}
    \ell(\gamma_r)\le K_N e^{-\sigma r}\ell(\gamma),\qquad 0\le r\le b.
\end{equation}
By the choice of \(K_N\),
\begin{equation}\label{shortes}
    \ell(\gamma_r)\le \|Df^{-r}\|_\infty\cdot \ell(\gamma)
    \le K_N e^{-\sigma r}\ell(\gamma)<\delta_L/2,\qquad 0\le r\le N.
\end{equation}
We prove by induction that
\begin{equation}\label{eq:long-curve-pullback}
    \ell(\gamma_r)\le e^{-\sigma r}\ell(\gamma),\qquad N\le r\le b.
\end{equation}
For \(r=N\), as a direct consequence of \eqref{shortes}, we have
\[
    d\bigl(f^j w,f^{b-N+j}x\bigr)\le \ell(\gamma_{N-j})<\delta_L/2,\qquad 0\le j\le N,\ w\in\gamma_N.
\]
Applying Lemma~\ref{lem:robust-suffix-conorm} to \(w\), we obtain \(A_N^E(w)\ge \sigma N\). Since the iterates of \(D\) are tangent to the cone field \(\mathcal{C}_\theta^E\),
\[
    m\bigl(Df_w^N|_{T_w f^{b-N}(D)}\bigr)\ge e^{\sigma N},\qquad w\in\gamma_N.
\]
Consequently, \(\ell(\gamma_N)\le e^{-\sigma N}\ell(\gamma)\). Suppose \eqref{eq:long-curve-pullback} holds up to \(r-1\), where \(r>N\). Then
\[
    \ell(\gamma_r)\le \|Df^{-1}\|_\infty e^{-\sigma(r-1)}\ell(\gamma)
    \le K_N e^{-\sigma r}\ell(\gamma)<\delta_L/2.
\]
Thus \(\gamma_0,\ldots,\gamma_r\) all lie in the corresponding \(\delta_L/2\)-balls. Using Lemma~\ref{lem:robust-suffix-conorm} again, we obtain \(\ell(\gamma_r)\le e^{-\sigma r}\ell(\gamma)\), which completes the induction, proving \eqref{eq:long-curve-pullback}. Together with \eqref{shortes}, this gives \eqref{eq:terminalpullback}.

We now prove \eqref{tls}. Given \(r_*>0\), choose \(b_*\ge N\) so that \(\rho_N e^{-\sigma b_*}<r_*\). Suppose that \(d_D(x,\partial D)\ge r_*\), \(b\in S_{c,L,N}^E(x)\), and \(b\ge b_*\), but
\[
    d_{f^b(D)}\bigl(f^b x,\partial f^b(D)\bigr)\le\rho_N.
\]
Choose a \(C^1\) curve \(\gamma:[0,1]\to\overline{f^b(D)}\) such that
\[
    \gamma(0)=f^b(x),\quad \gamma(1)\in\partial f^b(D),\quad \gamma([0,1))\subset f^b(D),\quad \ell(\gamma)\le\rho_N.
\]
For every \(0<t<1\), the long-curve estimate \eqref{eq:long-curve-pullback} applies to \(\gamma|_{[0,t]}\), since \(b\ge N\), giving
\[
    \ell\bigl(f^{-b}(\gamma|_{[0,t]})\bigr)\le e^{-\sigma b}\ell(\gamma|_{[0,t]}).
\]
Letting \(t\uparrow1\) and using \(f^{-b}(\gamma(1))\in\partial D\), we obtain
\[
    d_D(x,\partial D)\le \ell(f^{-b}(\gamma))\le e^{-\sigma b}\ell(\gamma)
    \le\rho_N e^{-\sigma b}\le \rho_N e^{-\sigma b_*}<r_*,
\]
a contradiction to our assumption. This proves \eqref{tls}.
\end{proof}

We henceforth fix \(\sigma=c/2\) and discuss bounded distortion
at sliding suffix times.
Write \(d=\dim E\), and let \(d_{\Gr}\) denote the distance
induced by a fixed smooth
Riemannian metric on \(\Gr_d(T\mathbf{M})\).
For \(d\)-planes \(P,Q\subset T_x\mathbf{M}\), we write
\[
 d_{\Gr}(P,Q):=d_{\Gr}\bigl((x,P),(x,Q)\bigr).
\]
For a \(C^1\) disk \(D\) of dimension \(d\)
and \(\eta>0\), define
\[
 \kappa_{\eta}(D)
 :=
 \sup_{\substack{x,y\in D\\0<d_D(x,y)<\delta_0}}
 \frac{d_{\Gr}\bigl(T_xD,\ D_y(\exp_x^{-1})(T_yD)\bigr)}{d_D(x,y)^{\eta}},
\]
where both planes in the numerator are viewed as subspaces
of \(T_x\mathbf{M}\).

The following follows from
\cite[Proposition~2.2, Remark~2.3 and Corollary~2.4]{ABV}.

\begin{lemma}\label{tangent-control}
There exists \(0<\alpha_*\le\alpha\) such that, for every
\(C^1\) disk \(D\subset U\) tangent to \(\cC_\theta^E\) with
\(\kappa_{\alpha_*}(D)<\infty\), we have
\[
 L_D:=\sup_{n\ge0}\kappa_{\alpha_*}(f^nD)<\infty.
\]
\end{lemma}

Fix \(\alpha_*\) as in Lemma~\ref{tangent-control}.
A \(C^1\) disk \(D\subset U\) tangent to \(\cC_\theta^E\)
is called an \emph{admissible \(E\)-disk} if
\(\kappa_{\alpha_*}(D)<\infty\).
In particular, every \(C^2\) embedded closed disk \(D\subset U\)
tangent to \(\cC_\theta^E\) is admissible.

As a consequence of Proposition \ref{prop:suffix-preball}, we have the following bounded distortion property.

\begin{corollary}\label{prop:bounded-distortion}
Let $c>0$ and $L\in \N$, and let \(D\subset U\)
be an admissible \(E\)-disk. Let $N_0$ be as given by Proposition \ref{prop:suffix-preball}, applied with \(\sigma=c/2\).  Then there exists a constant $\mathscr L=\mathscr L(c,D)\ge1$ such that, for every $N\ge N_0$,  every $x\in D$, every $b\in S^E_{c,L,N}(x)$ and every $0\le a <b$,
$$
 \frac{1}{\mathscr L}\le \frac{\left|\Jac (Df^{b-a}|_{T_{f^ay}f^aD})\right|}{\left|\Jac( Df^{b-a}|_{T_{f^ax}f^aD})\right|}\le \mathscr L
$$
for any $y\in D$ satisfying $d_{f^b(D)}(f^b(y),f^b(x))\le \rho_N$.
\end{corollary}

\begin{proof}
Fix \(N,x,b,a,y\) as in the statement.
By \cite[Corollary~2.4(c)]{ABV}, there exists $H_{\ast}=H_{\ast}(D)>0$ such that, for every $i\ge 0$, the function
$$
f^iD\ni y\longmapsto \log |\Jac(Df|_{T_y f^iD})|
$$
is $(H_{\ast}, \alpha_\ast)$-H\"older continuous. Thus,
$$
\begin{aligned}
\left|\log \frac{\left|\Jac (Df^{b-a}|_{T_{f^ay}f^aD})\right|}{\left|\Jac( Df^{b-a}|_{T_{f^ax}f^aD})\right|}\right| &\le \sum_{i=a}^{b-1}\left|\log |\Jac(Df|_{T_{f^i(y)} f^iD})|-\log |\Jac(Df|_{T_{f^i(x)} f^iD})|\right|\\
&\le \sum_{i=a}^{b-1} H_{\ast}\left(d_{f^iD}(f^i(y),f^i(x))\right)^{\alpha_\ast}.
\end{aligned}
$$
On the other hand, by Proposition \ref{prop:suffix-preball} (i) we have
$$
d_{f^iD}\left(f^i(y),f^i(x)\right)\le K_N{\rm e}^{-\frac{1}{2}(b-i)c}d_{f^bD}(f^b(y),f^b(x))\le {\rm e}^{-\frac{1}{2}(b-i)c}, \qquad 0\le i \le b-1,
$$
where we use $2K_N\rho_N<1$ from \eqref{fkv}.
Consequently,
$$
\left|\log \frac{\left|\Jac (Df^{b-a}|_{T_{f^ay}f^aD})\right|}{\left|\Jac( Df^{b-a}|_{T_{f^ax}f^aD})\right|}\right| \le \frac{H_{\ast}}{1-{\rm e}^{-c\alpha_{\ast}/2}}.
$$
To complete the proof, it suffices to take $\mathscr L(c,D)=\exp \bigl(H_{\ast}/(1-{\rm e}^{-c\alpha_{\ast}/2})\bigr)$.

\end{proof}

\subsubsection{Dynamical density}

Pugh and Shub \cite{PS00} introduced dynamical density points
along strong unstable manifolds, extending the classical notion
of Lebesgue density points, see \cite{Xia06,CYZ,AlvesPinheiro2008,BurguetYang} for further developments.
Here we establish an analogous density theorem for dynamical balls generated by
sliding suffix times.

We choose $q\ge 1$ such that, for every \(C^1\) disk
\(\Sigma\) tangent to \(\cC_\theta^E\) and every \(w\in\Sigma\),
\begin{equation}\label{eq:local-graph-volume}
 q^{-1}r^{\dim E}
 \le \Leb_\Sigma\bigl(B_\Sigma(w,r)\bigr)
 \le qr^{\dim E},
 \qquad
 0<r<\min\{\delta_0,d_\Sigma(w,\partial\Sigma)\}.
\end{equation}
For \(c>0\), \(L\ge1\) and \(\sigma=c/2\), fix \(N\ge N_0\), \(K_N\) and \(\rho_N\) from
Proposition~\ref{prop:suffix-preball}. For $r_*>0$ and an admissible disk $D$, we consider the subset
\begin{equation*}
 D^\infty
 :=\{x\in D:d_D(x,\partial D)\ge r_*,\
 S_{c,L,N}^E(x)\text{ is unbounded}\}.
\end{equation*}
Let \(b_*\) be the threshold in
Proposition~\ref{prop:suffix-preball}(ii) for boundary margin \(r_*\)
and terminal radius \(\rho_N\).
Choose \(R_N>0\) so that
\[
 (1+2K_N)R_N<\rho_N.
\]
For \(x\in D^\infty\) and
\(b\in S_{c,L,N}^E(x)\) with \(b\ge b_*\), define
$$
 V_b^D(x)
 :=(f^b|_D)^{-1}\bigl(B_{f^bD}(f^bx,R_N)\bigr), \quad \widehat V_b^D(x)
 :=(f^b|_D)^{-1}\bigl(B_{f^bD}(f^bx,\rho_N)\bigr).
$$

\begin{proposition}
\label{thm:dynamical-density-new}
Let \(A\subset D^\infty\) be measurable with \(\Leb_D(A)>0\).
Then, for \(\Leb_D\)-almost every \(x\in A\),
\[
 \lim_{\substack{b\to\infty\\b\in S_{c,L,N}^E(x)}}
 \frac{\Leb_D(A\cap V_b^D(x))}{\Leb_D(V_b^D(x))}=1.
\]
Such a point \(x\) is called a \emph{suffix dynamical density point}
of \(A\) in \(D\).
\end{proposition}

The main ingredients for the proof of Proposition \ref{thm:dynamical-density-new} are listed in the following lemma.

\begin{lemma}\label{three}
The following properties hold:
\begin{enumerate}[(i)]
\item\label{shrinking} For every \(x\in D^\infty\),
\(\Leb_D(\widehat V_b^D(x))\) tends to zero as $b\to\infty$ with $b\in S_{c,L,N}^E(x)$.
\item\label{density:volume} There exists \(C_N\ge1\) such that for any $x\in D^\infty$ and $b\in S_{c,L,N}^E(x)\cap [b_{\ast},+\infty)$,
$
 \Leb_D(\widehat V_b^D(x))\le C_N\Leb_D(V_b^D(x)).
$
\item\label{density:engulfing} If \(x_i\in D^\infty\), \(b_i\in S_{c,L,N}^E(x_i)\),
\(b_*\le b_1\le b_2\), and
\(V_{b_1}^D(x_1)\cap V_{b_2}^D(x_2)\ne\varnothing\), then
$$
 V_{b_2}^D(x_2)\subset\widehat V_{b_1}^D(x_1).
$$
\end{enumerate}
\end{lemma}

\begin{proof}
By Proposition~\ref{prop:suffix-preball}(i), we have
\(
 \widehat V_b^D(x)\subset B_D(x,K_N\rho_Ne^{-\sigma b}),
\)
proving \eqref{shrinking}. Now we show
\eqref{density:volume}.
By Corollary~\ref{prop:bounded-distortion} and change of variables,
\[
 \frac{\Leb_D(\widehat V_b^D(x))}{\Leb_D(V_b^D(x))}
 \le\mathscr L^2
 \frac{\Leb_{f^bD}(B_{f^bD}(f^bx,\rho_N))}
      {\Leb_{f^bD}(B_{f^bD}(f^bx,R_N))}
 \le\mathscr L^2q^2\left(\frac{\rho_N}{R_N}\right)^{\dim E}.
\]
This proves \ref{density:volume} with the last expression as \(C_N\).

For \eqref{density:engulfing}, take
\(w\in V_{b_1}^D(x_1)\cap V_{b_2}^D(x_2)\) and
\(y\in V_{b_2}^D(x_2)\).
Applying Proposition~\ref{prop:suffix-preball}(i) at \(x_2\)
to both \(y\) and \(w\), we obtain
\[
 \begin{aligned}
 d_{f^{b_1}D}(f^{b_1}y,f^{b_1}x_1)
 &\le d_{f^{b_1}D}(f^{b_1}y,f^{b_1}x_2)
      +d_{f^{b_1}D}(f^{b_1}x_2,f^{b_1}w)\\
 &\qquad+d_{f^{b_1}D}(f^{b_1}w,f^{b_1}x_1)\\
 &<\bigl(2K_Ne^{-\sigma(b_2-b_1)}+1\bigr)R_N
 <\rho_N.
 \end{aligned}
\]
Thus \(y\in\widehat V_{b_1}^D(x_1)\), proving
\eqref{density:engulfing}.
\end{proof}

Now we are ready to prove Proposition \ref{thm:dynamical-density-new}.

\begin{proof}[Proof of Proposition \ref{thm:dynamical-density-new}]
For any \(0<\eta<1\), we denote
\[
 A_\eta:=\left\{x\in A:
 \limsup_{\substack{b\to\infty\\b\in S_{c,L,N}^E(x)}}
 \frac{\Leb_D(V_b^D(x)\setminus A)}{\Leb_D(V_b^D(x))}
 >\eta\right\}.
\]
It suffices to show that \(\Leb_D(A_\eta)=0\).

Given \(\varepsilon>0\), choose an open set \(O\subset D\)
containing \(A\) with \(\Leb_D(O\setminus A)<\varepsilon\),
and let \(K\subset A_\eta\) be compact.
For each \(x\in K\), the definition of \(A_\eta\) and Proposition~\ref{prop:suffix-preball}(i) give a time \(b(x)\in S_{c,L,N}^E(x)\), with \(b(x)\ge b_*\), such that
\[
 V_{b(x)}^D(x)\subset O,
 \qquad
 \Leb_D(V_{b(x)}^D(x)\setminus A)
 >\eta\,\Leb_D(V_{b(x)}^D(x)).
\]
These open preballs cover \(K\). Choose a finite subcover
\[
 K\subset\bigcup_{j=1}^{s}V_j,
 \qquad
 V_j:=V_{b_j}^D(x_j),
 \qquad
 b_1\le\cdots\le b_s,
\]
and put \(\widehat V_j:=\widehat V_{b_j}^D(x_j)\).

Set \(\mathcal I_0=\{1,\ldots,s\}\). As long as
\(\mathcal I_{k-1}\ne\varnothing\), define
\[
 i_k:=\min\mathcal I_{k-1},
 \qquad
 \mathcal I_k
 :=
 \{j\in\mathcal I_{k-1}:V_j\cap V_{i_k}=\varnothing\}.
\]
Since \(i_k\notin\mathcal I_k\), the construction stops with
\(\mathcal I_m=\varnothing\) for some \(m\le s\).
By construction, we have
\begin{itemize}
\item $V_{i_k}\cap V_{i_\ell}=\varnothing$ whenever $k\ne \ell$;
\item $b_j\ge b_{i_k},\quad V_j\cap V_{i_k}\ne\varnothing$ for every \(j\in\mathcal I_{k-1}\setminus\mathcal I_k\).
\end{itemize}
Lemma \ref{three} (iii) therefore gives
\[
 V_j\subset\widehat V_{i_k},
 \qquad
 \forall\, j\in\mathcal I_{k-1}\setminus\mathcal I_k.
\]
Consequently,
\[
 K
 \subset\bigcup_{j=1}^{s}V_j
 =
 \bigcup_{k=1}^{m}
 \ \bigcup_{j\in\mathcal I_{k-1}\setminus\mathcal I_k}V_j
 \subset\bigcup_{k=1}^{m}\widehat V_{i_k}.
\]
Using Lemma \ref{three} (ii), the choice of the
\(V_j\), and the disjointness of the selected preballs, we obtain
\[
 \begin{aligned}
 \Leb_D(K)
 &\le\sum_{k=1}^{m}\Leb_D(\widehat V_{i_k})
 \le C_N\sum_{k=1}^{m}\Leb_D(V_{i_k})\\
 &<\frac{C_N}{\eta}
       \sum_{k=1}^{m}\Leb_D(V_{i_k}\setminus A)\\
 &=\frac{C_N}{\eta}
       \Leb_D\left(
       \left(\bigcup_{k=1}^{m}V_{i_k}\right)\setminus A
       \right)\\
 &\le\frac{C_N}{\eta}\Leb_D(O\setminus A)
 <\frac{C_N\varepsilon}{\eta}.
 \end{aligned}
\]
Letting \(\varepsilon\downarrow0\) gives \(\Leb_D(K)=0\).
Since \(K\subset A_\eta\) was arbitrary, inner regularity implies
\(\Leb_D(A_\eta)=0\). The proof is thus complete.
\end{proof}

\subsection{Multi-component Gibbs estimates}

We prove a Gibbs estimate on finite unions of time intervals
whose nonzero endpoints are sliding suffix times.
For a finite set \(Q\subset\mathbb Z_{\ge0}\), write its connected
components as \([a_i,b_i)_{\mathbb Z}\), and set
\[
 \partial Q:=\bigcup_i\{a_i,b_i\}.
\]

\begin{theorem}
\label{thm:multi-gibbs}
Fix \(c>0\), an integer \(L\ge1\), an admissible \(E\)-disk \(D\) of finite volume, and
\(r_*>0\). Let \(N\ge N_0\) be an integer, where \(N_0\) is given by
Proposition~\ref{prop:suffix-preball} with \(\sigma=c/2\).
For every \(\varepsilon>0\), there exist \(\delta>0\) and \(C\ge1\)
with the following property.

Let \(Q\subset\mathbb Z_{\ge0}\) be a nonempty finite set, and let
$$
\Gamma=\left\{y\in D: d_D(y,\partial D)\ge r_*,\quad
 \partial Q\setminus\{0\}\subset S_{c,L,N}^E(y)\right\}.
$$
Then, for every finite Borel partition \(\cP\) with
\(\diam\cP<\delta\) and every \(x\in D\),
\begin{equation*}
 \Leb_D\bigl(\Gamma\cap\cP^Q(x)\bigr)
 \le C^{\#\partial Q}
       \exp\left(\varepsilon\#Q-\sum_{j\in Q}\log\Jac\bigl(Df|_{T_{f^jx}f^jD}\bigr)\right).
\end{equation*}
\end{theorem}

\subsubsection{Jacobian comparison and local graphs}

We first establish a Jacobian comparison for cone-tangent disks
and record the local graph properties used in the proof.

\paragraph{Jacobian comparison.}
For a \(C^1\) disk \(D\subset U\) tangent to \(\cC_\theta^E\), write
\[
 J_j(w)=\log\Jac\left(Df_w|_{T_wf^jD}\right),\quad w\in f^jD,
\]
and, for a finite time set \(Q\), put
\[
 J_Q(y):=\sum_{j\in Q}J_j(f^jy),\qquad y\in D.
\]
The following estimate is uniform over all such disks and
all time intervals.

\begin{lemma}
\label{lem:jacobian-comparison}
For every \(\varepsilon>0\), there exist
\(\delta_\varepsilon>0\) and \(C_\varepsilon\ge1\) such that,
for every $C^1$ disk \(D\subset U\) tangent to $\cC_{\theta}^E$, every interval
\(I=[a,b)_{\mathbb Z}\) with \(0\le a<b\), and every \(y,z\in D\), if
\(
 d(f^jy,f^jz)<\delta_\varepsilon
\)
for all $a\le j < b$,
then
$$
 |J_I(y)-J_I(z)|\le\varepsilon(b-a)+\log C_\varepsilon.
$$
\end{lemma}

\begin{proof}
Set
\[
 \Psi(w,V):=\log\Jac(Df_w|_V),\quad w\in \overline{U},
 \qquad
 \gamma:=
 \max_{\substack{w\in\overline U\\
                  V\subset\cC_\theta^E(w),\ \dim V=d}}
 |\Psi(w,V)|.
\]
Fix \(\varepsilon>0\). By continuity of \(\Psi\) on the compact
Grassmann bundle \(\Gr_d(T\mathbf{M})|_{\overline U}\), there exists \(\eta>0\) such that
\[
 d_{\Gr}(V,W)<\eta
 \quad\Longrightarrow\quad
 |\Psi(w,V)-\Psi(w,W)|<\varepsilon/3
\]
for every \(w\in\overline U\) and all \(d\)-planes
\(V,W\subset T_w\mathbf{M}\). Since the $E$-direction cone field $\cC^E_{\theta}$ is contracted by domination, there exists \(m=m_\varepsilon\ge1\) such that for any $w\in \overline{U}$,
$
 d_{\Gr}(Df_w^mV,Df_w^mW)<\eta
$
whenever the $d$-planes $V, W$ are contained in $\cC^E_{\theta}(w)$. By the choice of $\eta$,
\begin{equation}\label{wvin}
\left|\Psi(f^m(w), Df_w^mV)-\Psi(f^m(w),Df_w^mW)\right|<\varepsilon/3
\end{equation}
for all $V, W \subset \cC^E_{\theta}(w)$ and $w\in \overline{U}$.

For this fixed \(m\), consider the function
\[
 g_m(w):=\Psi\bigl(f^mw,Df_w^m(\widetilde E(w))\bigr),
 \qquad w\in\overline U.
\]
Note that \(\widetilde E(w)\subset\cC_\theta^E(w)\), by taking $W=\widetilde E(w)$ in \eqref{wvin} we get
\begin{equation}\label{mwg}
 \bigl|\Psi(f^mw,Df_w^mV)-g_m(w)\bigr|<\varepsilon/3
\end{equation}
holds for every $d$-plane $V\subset\cC_\theta^E(w)$.
Since \(g_m\) is continuous on the compact set \(\overline U\),
there exists \(\delta_\varepsilon>0\) such that, for every
\(w,v\in\overline U\),
\begin{equation}\label{delt}
 d(w,v)<\delta_\varepsilon
 \quad\Longrightarrow\quad
 |g_m(w)-g_m(v)|<\varepsilon/3.
\end{equation}

Now let \(D\) be a $C^1$ disk tangent to $\cC^E_{\theta}$, \(I=[a,b)_{\mathbb Z}\), and \(y,z\in D\)
satisfy
\[
 d(f^jy,f^jz)<\delta_\varepsilon,\qquad  a\le j<b.
\]
Write \(D_j=f^jD\) for $j\ge 0$.
Assume first that $m<b-a$.
For \(a+m\le j<b\), put $ w=f^{j-m}y$ and $v=f^{j-m}z$.
Observe that
$$
T_{f^jy}D_j=Df_w^m(T_wD_{j-m}),\qquad T_{f^jz}D_j=Df_v^m(T_vD_{j-m}),
$$
and $T_wD_{j-m}$ and $T_vD_{j-m}$ are contained in $\cC^E_{\theta}(w)$ and $\cC^E_{\theta}(v)$, respectively. By \eqref{mwg},
it follows that
$$
 |J_j(f^jy)-g_m(w)|<\varepsilon/3,\qquad |J_j(f^jz)-g_m(v)|<\varepsilon/3.
$$
Since \(j-m\in I\), we have \(d(w,v)<\delta_\varepsilon\) by assumption, which together with \eqref{delt} implies
$
|g_m(w)-g_m(v)|<\varepsilon/3.
$
Altogether, we obtain
\[
 \begin{aligned}
 |J_j(f^jy)-J_j(f^jz)|
 &\le |J_j(f^jy)-g_m(w)|
      +|g_m(w)-g_m(v)|\\
 &\qquad +|g_m(v)-J_j(f^jz)|<\varepsilon.
 \end{aligned}
\]
In view of the choice of $\gamma$, this gives
\[
 \begin{aligned}
 |J_I(y)-J_I(z)|
 &\le
 \sum_{j=a}^{a+m-1}|J_j(f^jy)-J_j(f^jz)|+\sum_{j=a+m}^{b-1}|J_j(f^jy)-J_j(f^jz)|\\
 &\le 2 m \gamma
 +\varepsilon (b-a-m)\\
 &\le 2m\gamma+\varepsilon(b-a).
 \end{aligned}
\]
If $m\ge b-a$, then
$$
|J_I(y)-J_I(z)|\le (b-a)2\gamma\le 2m\gamma<2m\gamma+\varepsilon(b-a).
$$
Thus the conclusion holds with
\(
 C_\varepsilon:=\exp(2m_\varepsilon\gamma).
\)
\end{proof}

\paragraph{Invariance of local graphs.}
For a radius \(0<\delta < \delta_0\), use the local
exponential chart at \(x\in U\) to define
\[
 \mathscr F_x(\delta):=
 \{\Delta=\exp_x(\operatorname{graph}\psi)\subset B(x,\delta)\},
\]
where \(\Omega\subset\widetilde E(x)\) is an open neighborhood of \(0\),
\(\psi\in C^1(\Omega,\widetilde F(x))\), \(\psi(0)=0\), and
\(\operatorname{Lip}(\psi)\le2\theta\).

By the continuity and the invariance of the cone fields, we get the following result.

\begin{lemma}\label{lem:local-graphs}
There exists $0 < \delta_1 < \delta_0$ such that the families $\mathscr{F}_x:=\mathscr{F}_x(\delta_1)$ are locally invariant:
\begin{equation*}
  f(\Delta) \cap B(fx,\delta_1) \in \mathscr F_{fx}
  \qquad \text{whenever $x \in U$ and $\Delta \in \mathscr F_x$.}
\end{equation*}
Moreover, for every $C^1$ disk $D$ tangent to $\cC_{\theta}^E$ and every $x \in D$, we have
$B_D(x,\delta) \in \mathscr F_x$ provided that
$0 < \delta < \min\{\delta_1, \frac{1}{2} d_D(x,\partial D)\}$.
\end{lemma}

For \(x\in U\), define the local transverse fiber
$$
 \mathcal T_x:=\exp_x(\widetilde F(x))\cap B(x,\delta_1).
$$
By local graph structure and uniform transversality, there exists \(C_0\ge 4\) such that,
for every \(C^1\) disk \(\Sigma\subset U\) tangent to
\(\cC_\theta^E\), with center \(w\) and intrinsic radius
\(0<r\le\delta_1\), and every \(z\in U\) satisfying \(d(w,z)<r/C_0\),
the intersection \(\Sigma\cap\mathcal T_z\) consists of a unique point \(p\), with
\begin{equation*}
 d_\Sigma(w,p)\le C_0d(w,z).
\end{equation*}

\subsubsection{One-component geometry and gap transfer}

Throughout, we fix constants \(c,L,r_*,N\) as in
Theorem~\ref{thm:multi-gibbs}.

\paragraph{Choice of the working radius.}
Take \(K_N\ge1\) and \(\rho_N>0\) from
Proposition~\ref{prop:suffix-preball}, with \(\sigma=c/2\).
Decrease the radius \(R_N\) used to define the suffix preballs,
if necessary, so that
$$
 C_0(1+K_N)R_N<\min\{\rho_N,\delta_1\}
$$
and
$$
 d_{f^tD}(f^ty,\partial(f^tD))>C_0(1+K_N)R_N
$$
whenever $D$ is a $C^1$ disk tangent to $\cC^E_{\theta}$, $y\in D$ and satisfies
\[
 d_D(y,\partial D)\ge r_*,
 \qquad t\in\{0\}\cup S_{c,L,N}^E(y).
\]
Such a choice is possible: Proposition~\ref{prop:suffix-preball}(ii)
gives the uniform lower bound \(\rho_N\) for all sufficiently large
suffix times, while the remaining finitely many times satisfy
\[
 d_{f^tD}(f^ty,\partial(f^tD))
 \ge \frac{r_*}{\|Df^{-t}\|_\infty}.
\]
Set
\begin{equation}\label{dgeo}
 \delta_{\rm geo}:=\frac{R_N}{2K_NC_0}.
\end{equation}

\begin{lemma}
\label{lem:component-geometry}
Let \(D\) be a $C^1$ disk tangent to $\cC^E_{\theta}$, and let \(y, z \in D\) satisfy
\(\min\{d_D(y,\partial D),d_D(z,\partial D)\}\ge r_*\).
Suppose that \(0 \le a \le b\) are integers satisfying $\{a,b\}\setminus\{0\}
 \subset S_{c,L,N}^E(y)\cap S_{c,L,N}^E(z)$.
If
\[
  d_{f^a D}(f^a y, f^a z) < 2R_N
  \quad \text{and} \quad
  \max_{a\le j \le b}d(f^j y, f^j z) < \delta_{\rm geo},
\]
then
\begin{equation*}
 \max_{a\le j\le b}
 d_{f^jD}(f^jy,f^jz)<R_N.
\end{equation*}
\end{lemma}

\begin{proof}
Write \(D_j=f^jD\), \(y_j=f^jy\), and \(z_j=f^jz\) for every \(j\ge0\). By the choice of $R_N$, $D_b$ contains the intrinsic ball \(B_{D_b}(z_b, R_N)\).
Since
$R_N<\delta_1$ and $d(z_b,y_b)<\delta_{\rm geo}<R_N/C_0$,
the ball \(B_{D_b}(z_b,R_N)\) meets the local transversal
\(\mathcal T_{y_b}\) in a unique point, which we denote by \(p_b\).
Moreover,
\begin{equation}\label{eq:pb-close}
d_{D_b}(p_b,z_b)
\le C_0 d(y_b,z_b)
<C_0\delta_{\rm geo}<R_N.
\end{equation}

First suppose \(a=b\). Then \(D_a=D_b\). From the hypothesis
\(d_{D_b}(y_b,z_b)<2R_N\) and \eqref{eq:pb-close}, we get
\[
 d_{D_b}(y_b,p_b)
 \le d_{D_b}(y_b,z_b)+d_{D_b}(z_b,p_b)
 <2R_N+C_0\delta_{\rm geo}<3R_N.
\]
Noting that
\(
 3R_N<
 \min\{\delta_1,\frac12 d_{D_b}(y_b,\partial D_b)\},
\)
Lemma~\ref{lem:local-graphs} gives
\(B_{D_b}(y_b,3R_N)\in\mathscr F_{y_b}\).
Therefore,
\[
 p_b\in B_{D_b}(y_b,3R_N)\cap\mathcal T_{y_b}
 =\{y_b\},
\]
and hence
\(
 d_{D_b}(y_b,z_b)<C_0\delta_{\rm geo}<R_N.
\)
So the conclusion holds when \(a=b\), including \(a=b=0\).

Now assume $0\le a<b$.
For $a \le j \le b$, define $p_j := f^{j-b}(p_b)$.
Recalling $R_N<\rho_N$, it follows from Proposition~\ref{prop:suffix-preball}(i) and \eqref{eq:pb-close} that
\begin{equation}\label{eq:auxiliary-point-pullback}
  d_{D_j}(p_j,z_j) < K_Ne^{-\frac c2(b-j)}C_0\delta_{\rm geo} \le K_NC_0\delta_{\rm geo}=R_N/2, \qquad a\le j\le b.
\end{equation}
As a consequence, for any \(a\le j\le b\),
\begin{equation}\label{closeyp}
  d(y_j,p_j) \le d(y_j,z_j)+d_{D_j}(z_j,p_j) < (1+K_NC_0)\delta_{\rm geo} < \delta_1.
\end{equation}
Taking \(j=a\) in \eqref{eq:auxiliary-point-pullback} and using
\(d_{D_a}(y_a,z_a)<2R_N\), we get
\[
  d_{D_a}(y_a,p_a)\le d_{D_a}(y_a,z_a)+d_{D_a}(p_a,z_a)< 2R_N+R_N/2< 3R_N.
\]
Using \(3R_N < \min\{\delta_1,\frac12 d_{D_a}(y_a,\partial D_a)\}\) and Lemma~\ref{lem:local-graphs}, we obtain
$
p_a\in B_{D_a}(y_a,3R_N)\in \mathscr{F}_{y_a}.
$

Setting
\[
  \Delta_a:=B_{D_a}(y_a,3R_N), \qquad
  \Delta_{j+1}:=f(\Delta_j)\cap B(y_{j+1},\delta_1), \quad a\le j<b.
\]
Due to \eqref{closeyp}, we can use the invariance of local graphs inductively, showing
$
p_{b}\in \Delta_{b}\in \mathscr{F}_{y_{b}}.
$
As $\{p_b\}=B_{D_b}(z_b,R_N)\cap \mathcal T_{y_b}$, it follows that
$
p_b\in \Delta_{b}\cap \mathcal T_{y_b}=\{y_b\}.
$
Hence, $y_b=p_b$, and thus \(p_j=y_j\) for all \(a\le j\le b\). Finally, using
\eqref{eq:auxiliary-point-pullback},
\[
 \max_{a\le j\le b}d_{D_j}(y_j,z_j)
 <R_N/2<R_N.
\]
The proof is complete.
\end{proof}

For the volume estimates below, set
\[
 v_N:=q^{-1}R_N^{\dim E},
 \qquad
 V_N:=qR_N^{\dim E},
\]
where $q$ is given in \eqref{eq:local-graph-volume}.
The next lemma controls volume across a gap by a constant
independent of its length.

\begin{lemma}
\label{lem:gap-transfer}
Let $D$ be an admissible \(E\)-disk of finite volume.
Let \(0\le b<a\) be integers, and let \(\Gamma\subset D\)
be a nonempty measurable set with
\(\diam(f^a\Gamma)<\delta_{\rm geo}\) such that
\[
 \forall\, y\in \Gamma, \qquad d_D(y,\partial D)\ge r_*,
 \qquad a\in S_{c,L,N}^E(y).
\]
If \(b>0\), assume also that, for some \(y_0\in\Gamma\),
\begin{equation*}
 b\in S_{c,L,N}^E(y_0),
 \qquad
 f^b\Gamma\subset B_{f^bD}(f^by_0,R_N).
\end{equation*}
Then
$$
 \Leb_{f^bD}(f^b\Gamma)
 \le C_{\rm gap}\sup_{y\in\Gamma}
 \Leb_{f^aD}\bigl(
 f^a\Gamma\cap B_{f^aD}(f^ay,R_N)
 \bigr),
$$
where \(C_{\rm gap}\ge1\) is independent of \(a,b\) and \(\Gamma\).
\end{lemma}

\begin{proof}
Write \(D_j=f^jD\) for every \(j\ge0\). For every $y\in \Gamma$, as $
 d_D(y,\partial D)\ge r_*$ and $a\in  S_{c,L,N}^E(y)$, the choice of $R_N$ guarantees that the intrinsic ball $B_{D_a}(f^ay, R_N)$ is well-defined. Recall that
\begin{equation}\label{vn}
 \Leb_{D_a}(B_{D_{a}}(f^ay, R_N))\ge v_N \qquad \text{for any } y\in \Gamma.
 \end{equation}

Since $D_a$ has finite volume, we can select a maximal collection of pairwise disjoint balls $\{\Delta_i := B_{D_a}(f^ay_i, R_N)\}_{i=1}^k$ centered at points $y_i \in \Gamma$. We claim that
$$
f^a\Gamma\subset \bigcup_{1\le i \le k}\Delta_i.
$$
To this end, take any  \(y\in\Gamma\), the maximality gives an index \(i\) such that
\(
 B_{D_a}(f^ay,R_N)\cap\Delta_i\ne\varnothing.
\)
Hence \(d_{D_a}(f^ay,f^ay_i)<2R_N\). By applying Lemma~\ref{lem:component-geometry}, we get
\(d_{D_a}(f^ay,f^ay_i)<R_N\), which means $f^ay\in \Delta_i$. Hence the above covering property follows.

Put \(W_i:=f^{b-a}\Delta_i\) for every $1\le i \le k$. By construction, $W_i$ are pairwise disjoint and satisfy
\begin{equation}\label{coverw}
f^b\Gamma\subset \bigcup_{1\le i \le k}W_i.
\end{equation}
If \(b>0\), Proposition~\ref{prop:suffix-preball}(i)
and our assumption on $y_0$ imply
\[
 W_i
 \subset B_{D_b}(f^by_i,K_NR_N)
 \subset B_{D_b}\left(f^by_0,(1+K_N)R_N\right),\qquad 1\le i \le k.
\]
If \(b=0\), then all \(W_i\) lie in \(D\).
By \eqref{eq:local-graph-volume}, the constant
\[
 C_{\rm vol}:=1+\Leb_D(D)+q\bigl((1+K_N)R_N\bigr)^{\dim E}
\]
satisfies
\begin{equation}\label{c11}
 \sum_{i=1}^k\Leb_{D_b}(W_i)\le C_{\rm vol}.
\end{equation}
Since \(a\in S_{c,L,N}^E(y_i)\) and \(R_N<\rho_N\), by Corollary~\ref{prop:bounded-distortion}, it follows that
\[
 \frac{\Leb_{D_b}(f^b\Gamma\cap W_i)}
      {\Leb_{D_b}(W_i)}
 \le \mathscr L^2
 \frac{\Leb_{D_a}(f^a\Gamma\cap\Delta_i)}
      {\Leb_{D_a}(\Delta_i)}.
\]
Summing over $i$ and using \eqref{vn}, \eqref{coverw}, \eqref{c11}, it yields
$$
\begin{aligned}
  \Leb_{D_b}(f^b\Gamma) &\le \sum_{i=1}^k \Leb_{D_b}(f^b\Gamma \cap W_i) \\
  &\le \frac{\mathscr{L}^2}{v_N} \left( \sup_{y \in \Gamma} \Leb_{D_a}(f^a\Gamma \cap B_{D_a}(f^ay, R_N)) \right) \sum_{i=1}^k \Leb_{D_b}(W_i) \\
  &\le \frac{\mathscr{L}^2 C_{\rm vol}}{v_N} \sup_{y \in \Gamma} \Leb_{D_a}\bigl(f^a\Gamma \cap B_{D_a}(f^ay, R_N)\bigr).
\end{aligned}
$$
The conclusion follows with
$
 C_{\rm gap}:=\max\{1,\mathscr L^2C_{\rm vol}/v_N\}.
$
\end{proof}

\subsubsection{Proof of Theorem~\ref{thm:multi-gibbs}}

\begin{proof}[Proof of Theorem~\ref{thm:multi-gibbs}]
Fix \(\varepsilon>0\), and let \(\delta_\varepsilon, C_\varepsilon\) be the constants given by Lemma~\ref{lem:jacobian-comparison}. Set \(\widehat{\delta} := \min\{\delta_\varepsilon, \delta_{\rm geo}\}\), where \(\delta_{\rm geo}\) is specified in \eqref{dgeo}.
Choose \(\delta < \widehat{\delta}\) small enough such that
\[
  d(fu, fv) < \widehat{\delta} \quad \text{whenever } d(u, v) < \delta.
\]
Let \(\cP\) be a finite Borel partition with \(\diam \cP < \delta\), and let
\[
  Q = \bigcup_{i=1}^k I_i, \qquad I_i = [a_i, b_i)_{\mathbb Z}, \qquad b_i < a_{i+1}, \quad 1 \le i < k.
\]
Fix \(x\in D\), let \(\Gamma\) be as in the statement,
and write \(D_j:=f^jD\) for \(j\ge0\).
We may assume that \(\Gamma\cap\cP^Q(x)\ne\varnothing\). By the choice of \(\delta\),
\begin{equation}\label{mclose}
 \max_{a_i\le j\le b_i}d(f^jy,f^jz)<\widehat{\delta}
 \qquad
 \text{for all } y,z\in\cP^Q(x) \text{ and }  1\le i\le k.
\end{equation}
For every \(y \in \Gamma \cap \cP^Q(x)\) and \(1 \le i \le k\), define
\[
  A_i(y) := \{z \in \Gamma \cap \cP^Q(x) : d_{D_{a_i}}(f^{a_i}z, f^{a_i}y) < R_N\}, \quad
  M_i := \sup_{y \in \Gamma \cap \cP^Q(x)} \Leb_{D_{a_i}}\bigl(f^{a_i}A_i(y)\bigr).
\]
Recall that \(C_{\rm gap}\) is the constant given by
Lemma~\ref{lem:gap-transfer}.

\paragraph{Claim 1.} \textit{We have} \(\Leb_D(\Gamma \cap \cP^Q(x)) \le C_{\rm gap} M_1\).

\begin{proof}[Proof of Claim 1]
If \(a_1 > 0\), the statement follows directly from applying Lemma~\ref{lem:gap-transfer} to \(\Gamma \cap \cP^Q(x)\) between times \(0\) and \(a_1\).

If \(a_1 = 0\), choose a maximal collection of pairwise disjoint intrinsic balls \(\{B_D(y_\ell, R_N)\}_{\ell=1}^m\) centered at points \(y_\ell \in \Gamma \cap \cP^Q(x)\). Since each ball has volume at least \(v_N=q^{-1}(R_N)^{\dim E}\), we conclude
\[
 m\le\Leb_D(D)/v_N\le C_{\rm gap},
\]
by the choice of \(C_{\rm gap}\) in the proof of
Lemma~\ref{lem:gap-transfer}.
By the maximality of this collection, for every \(z \in \Gamma \cap \cP^Q(x)\) there exists \(1\le \ell \le m\) such that \(d_D(z, y_\ell) < 2R_N\). Since \(d(z, y_\ell) < \delta < \delta_{\rm geo}\) by the choice of \(\delta\), Lemma~\ref{lem:component-geometry} with \(a = b = 0\) implies \(d_D(z, y_\ell) < R_N\). Thus, these balls cover \(\Gamma \cap \cP^Q(x)\), and we obtain
\[
  \Leb_D\left(\Gamma \cap \cP^Q(x)\right) \le \sum_{\ell=1}^m \Leb_D\bigl(\Gamma \cap \cP^Q(x) \cap B_D(y_\ell, R_N)\bigr) \le m M_1 \le C_{\rm gap} M_1,
\]
which proves Claim 1.
\end{proof}

\paragraph{Claim 2.} \textit{We have}
\[
  M_1 \le V_N C_\varepsilon^k C_{\rm gap}^{k-1} \exp\left( \varepsilon \#Q - \sum_{i=1}^k J_{I_i}(x) \right).
\]

\begin{proof}[Proof of Claim 2]
For each \(1 \le i \le k\), \eqref{mclose} and Lemma~\ref{lem:jacobian-comparison} guarantee that
\[
  J_{I_i}(z) \ge J_{I_i}(x) - \varepsilon(b_i - a_i) - \log C_\varepsilon \qquad \text{for all } z \in \Gamma \cap \cP^Q(x).
\]
Consequently, for any \(y \in \Gamma \cap \cP^Q(x)\), a change of variables yields
\begin{equation}\label{aibi}
  \begin{aligned}
    \Leb_{D_{a_i}}\bigl(f^{a_i}A_i(y)\bigr)
    &= \int_{f^{b_i}A_i(y)} \Jac\bigl(Df^{-(b_i-a_i)}|_{T_u D_{b_i}}\bigr) \, d\Leb_{D_{b_i}}(u) \\
    &= \int_{f^{b_i}A_i(y)} e^{-J_{I_i}(f^{-b_i}u)} \, d\Leb_{D_{b_i}}(u) \\
    &\le C_\varepsilon e^{\varepsilon(b_i - a_i) - J_{I_i}(x)} \Leb_{D_{b_i}}\bigl(f^{b_i}A_i(y)\bigr).
  \end{aligned}
\end{equation}
By Lemma~\ref{lem:component-geometry}, for each \(1 \le i \le k\) and \(y \in \Gamma \cap \cP^Q(x)\), we have the inclusion
\begin{equation}\label{eq:gibbs-component-ball}
  f^{b_i}A_i(y) \subset B_{D_{b_i}}(f^{b_i}y, R_N).
\end{equation}

For \(1 \le i < k\), we verify the hypotheses of Lemma~\ref{lem:gap-transfer} for \(A_i(y)\) between times \(b_i\) and \(a_{i+1}\):
\begin{itemize}
  \item Every \(z \in A_i(y) \subset \Gamma\) satisfies \(d_D(z, \partial D) \ge r_*\) and \(a_{i+1} \in S_{c,L,N}^E(z)\).
  \item Since \(a_{i+1}\in Q\), we have
\(
 f^{a_{i+1}}A_i(y)\subset\cP(f^{a_{i+1}}x)\), thus \(\diam\bigl(f^{a_{i+1}}A_i(y)\bigr)
 <\delta<\delta_{\rm geo}\).
  \item The point \(y \in A_i(y)\) satisfies \(b_i \in S_{c,L,N}^E(y)\) and \eqref{eq:gibbs-component-ball}.
\end{itemize}
Applying Lemma~\ref{lem:gap-transfer} then yields
\[
  \begin{aligned}
    \Leb_{D_{b_i}}\bigl(f^{b_i}A_i(y)\bigr)
    &\le C_{\rm gap} \sup_{z \in A_i(y)} \Leb_{D_{a_{i+1}}}\left( f^{a_{i+1}}A_i(y) \cap B_{D_{a_{i+1}}}(f^{a_{i+1}}z, R_N) \right) \\
    &= C_{\rm gap} \sup_{z \in A_i(y)} \Leb_{D_{a_{i+1}}}\left( f^{a_{i+1}}\bigl(A_i(y) \cap A_{i+1}(z)\bigr) \right) \\
    &\le C_{\rm gap} \sup_{z \in A_i(y)} \Leb_{D_{a_{i+1}}}\bigl( f^{a_{i+1}}A_{i+1}(z) \bigr) \le C_{\rm gap} M_{i+1}.
  \end{aligned}
\]

For \(i = k\), \eqref{eq:gibbs-component-ball} directly bounds the volume by
\[
\Leb_{D_{b_k}}\bigl(f^{b_k}A_k(y)\bigr) \le \Leb_{D_{b_k}}\bigl(B_{D_{b_k}}(f^{b_k}y,R_N)\bigr)\le V_N.
\]
These together with \eqref{aibi}, yield
\[
  M_i \le C_\varepsilon C_{\rm gap} e^{\varepsilon(b_i - a_i) - J_{I_i}(x)} M_{i+1} \quad (1 \le i < k), \qquad
  M_k \le C_\varepsilon V_N e^{\varepsilon(b_k - a_k) - J_{I_k}(x)}.
\]
Iterating these inequalities proves Claim 2.
\end{proof}

Combining the two claims, we obtain
\[
 \Leb_D\left(\Gamma \cap \cP^Q(x)\right)
 \le V_N(C_\varepsilon C_{\rm gap})^k
 \exp\left(
 \varepsilon\#Q-
 \sum_{j\in Q}\log\Jac(Df|_{T_{f^jx}D_j})
 \right).
\]
Setting \(C := \max\{1, V_N, C_\varepsilon C_{\rm gap}\}\) and noting that \(\#\partial Q = 2k\), we have
\[
  V_N (C_\varepsilon C_{\rm gap})^k \le C^{k+1} \le C^{2k} = C^{\#\partial Q} \quad \text{for all } k \ge 1.
\]
This completes the proof of Theorem~\ref{thm:multi-gibbs}.
\end{proof}

\section{Empirical basin production}
\label{sec:marked-basin}

We prove the following leafwise version of Theorem~\ref{thm:B}.

\begin{theorem}
\label{thm:lvl-basin-production}
Let \(D\subset U\) be an admissible \(E\)-disk and let
\(\Gamma\subset D\) be measurable with \(\Leb_D(\Gamma)>0\).
Suppose that, for \(\Leb_D\)-almost every \(x\in\Gamma\), there is
\(\mu\in\pw(x)\) such that
\begin{equation}\label{eq:lvl-joint-signs}
 \lambda_E^-(y)>0\quad\text{for \(\mu\)-almost every }y,
 \qquad
 \mu\{y:\lambda_F^+(y)<0\}>0.
\end{equation}
Then there exists an ergodic hyperbolic SRB measure \(\mu_*\), supported
on \(\Lambda\), such that
\[
 \lambda_E^-(\mu_*)>0>\lambda_F^+(\mu_*),
 \qquad
 \Leb_D\bigl(\Gamma\cap\Basin(\mu_*)\bigr)>0.
\]
\end{theorem}

Theorem \ref{thm:B} can be deduced from Theorem \ref{thm:lvl-basin-production} as follows.

\begin{proof}[Proof of Theorem~\ref{thm:B}]
Let \(\mathcal S\) be the at most countable family of ergodic
hyperbolic SRB measures \(\nu\), supported on \(\Lambda\), satisfying
\(
 \lambda_E^-(\nu)>0>\lambda_F^+(\nu).
\)
It suffices to show the measurable subset
\[
 \Gamma_0:=
 \Gamma\setminus\bigcup_{\nu\in\mathcal S}\Basin(\nu)
\]
admits zero Lebesgue measure.

Assume by contradiction that \(\Leb(\Gamma_0)>0\).
Since $\Gamma_0\subset B(\Lambda)=\bigcup_{\ell \ge 0}f^{-\ell}(U)$, there exists $\ell\ge 0$ such that
$
\Leb\left(f^{\ell}(\Gamma_0)\cap U\right)>0.
$
By Fubini's theorem, there is a \(C^2\) embedded closed disk
\(D\subset U\), tangent to \(\cC_\theta^E\), such that
\(\widehat\Gamma:=D\cap f^\ell(\Gamma_0)\) has positive leaf volume.
The disk \(D\) is admissible, and
\(\pw(f^\ell x)=\pw(x)\), so \(\widehat\Gamma\) satisfies the
hypotheses of Theorem~\ref{thm:lvl-basin-production}.
Thus there exists \(\nu\in\mathcal S\) such that
$$
\Leb_D\left(\widehat{\Gamma}\cap \Basin(\nu)\right)>0.
$$
In particular, we have $f^{\ell}(\Gamma_0)\cap \Basin(\nu)\neq \varnothing$. Using the invariance of \(\Basin(\nu)\), we get
$$
\Gamma_0\cap \Basin(\nu)\neq \varnothing,
$$
a contradiction to the definition of $\Gamma_0$. So the proof of Theorem~\ref{thm:B} is complete.

\end{proof}

\paragraph{Preliminary reductions.}
By restricting to a relatively compact admissible subdisk meeting
\(\Gamma\) in positive volume, we may assume that
\(\Leb_D(D)<\infty\). After further restricting \(\Gamma\), we fix
\(r_*>0\) such that
\begin{equation*}
 d_D(x,\partial D)\ge2r_*
 \qquad \text{for every } x\in\Gamma,
\end{equation*}
and assume that the empirical hypothesis holds at every point of
\(\Gamma\).

\subsection{F\o lner selection}

For a finite set \(T\subset\mathbb Z_{\ge0}\) and an integer \(p\ge1\), define
\begin{equation*}
 T^{[p]}:=
 \bigcup_{\substack{a,b\in T\\0<b-a\le p}}[a,b)_{\mathbb Z}.
\end{equation*}
One can check by construction that, for $T\subset[0,n)_{\mathbb Z}$ and $p\ge 1$:
\begin{equation}\label{tpes}
\partial T^{[p]}\subset T,\qquad \#\partial T^{[p]}\le 2n/p+2, \qquad \#(T\setminus T^{[p]})\le n/p+1.
\end{equation}
Indeed, the components of \(T^{[p]}\) have endpoints in \(T\)
and are separated by gaps longer than \(p\), so there are at most
\(n/p+1\) components.
Likewise, any two distinct points of \(T\setminus T^{[p]}\)
are more than \(p\) apart, giving the last bound.

\begin{proposition}
\label{prop:hyperbolic-folner-selection}
There exist \(c,\eta,\rho>0\), integers \(L,\ell\ge1\) and \(N\ge N_0\),
a measurable set \(\Gamma^*\subset\Gamma\) of positive volume,
sequences of positive integers \(M_j,n_j\to\infty\) with
\(n_j\ge jM_j\), measurable sets \(E_j\subset\Gamma^*\), and
nonempty finite sets \(Q_j\subset[0,n_j)_{\mathbb Z}\) such that,
writing
\[
 \mathcal T_j(x):=
 S_{c,L,N}^E(x)\cap G_{\eta,\ell,M_j}^{F}(x)
 \cap[0,n_j),
\]
the following hold:
\begin{enumerate}[(i)]
\item\label{com1}
\[
\lim_{j\to \infty} \frac1{n_j}\log\Leb_D(E_j)=0.
\]
\item We have \(\#\partial Q_j/n_j\to0\), and, for every \(x\in E_j\),
\[
 \partial Q_j\subset\mathcal T_j(x),\qquad
 \#(Q_j\cap\mathcal T_j(x))\ge\rho n_j.
\]
\item\label{com3}
\begin{equation}\label{eq:prop-filling}
 \lim_{p\to\infty}\limsup_{j\to\infty}\sup_{x\in E_j}
 \frac{\#(Q_j\setminus\mathcal T_j(x)^{[p]})}{n_j}=0.
\end{equation}
\item\label{com4} For the corresponding suffix preballs,
\begin{equation}\label{eq:uniform-density}
 \lim_{B\to\infty}
 \sup_{\substack{x\in\Gamma^*\\
                  b\in S_{c,L,N}^E(x),\ b\ge B}}
 \frac{\Leb_D(V_b^D(x)\setminus\Gamma)}{\Leb_D(V_b^D(x))}=0.
\end{equation}
\end{enumerate}
Here \(N_0\) is given by Proposition~\ref{prop:suffix-preball}
with \(\sigma=c/2\).
\end{proposition}

We first adapt the filling and selection arguments of \cite[Lemmas~1--2]{BurguetSurface}. The main difference here is that we will deal with finite subsets $T_j(x)\subset [0,n_j)_{\mathbb Z}$ that depend on the sequence index $j$.

\begin{lemma}
\label{lem:finite-folner-selection}
Let \((X,\lambda)\) be a finite measure space, let \(n_j\) be positive
integers tending to infinity, and let \(B_j\subset X\) be measurable
sets of positive measure satisfying
\[
\lim_{j\to \infty} \frac1{n_j}\log\lambda(B_j)=0.
\]
Assume that \(x\mapsto T_j(x)\subset[0,n_j)_{\mathbb Z}\) is measurable
on \(B_j\) and, for some \(d>0\), $\#T_j(x)\ge dn_j$ for every $x\in B_j$.
Then, after passing to a subsequence, there exist measurable sets
\(A_j\subset B_j\) of positive measure and nonempty finite sets
\(Q_j\subset[0,n_j)_{\mathbb Z}\), independent of \(x\), such that
\begin{enumerate}[(i)]
\item\label{mAj} \(n_j^{-1}\log\lambda(A_j)\to0\), \(\liminf_{j\to\infty}\#Q_j/n_j\ge d\);
\item\label{bQj} \(\partial Q_j\subset T_j(x)\) for every \(x\in A_j\), and
\(\#\partial Q_j/n_j\to0\);
\item\label{tpc} $T_j$ is dense in $Q_j$ on $A_j$, i.e.,
$$
 \lim_{p\to\infty}\limsup_{j\to\infty}\sup_{x\in A_j}
 \frac{\#(Q_j\setminus T_j(x)^{[p]})}{n_j}=0.
 $$
\item \label{QcapT}
$$
 \liminf_{j\to\infty}\inf_{x\in A_j}
 \frac{\#(Q_j\cap T_j(x))}{n_j}\ge d.
$$
\end{enumerate}
\end{lemma}

\begin{proof}
For \(x\in B_j\) and \(p\ge1\), set
\[
 a_{j,p}(x):=\frac{\#T_j(x)^{[p]}}{n_j}\in[0,1].
\]
These functions are nondecreasing in \(p\), and \eqref{tpes} together with our assumption gives
\begin{equation}\label{tjn}
 a_{j,p}(x)
 \ge
 \frac{\#T_j(x)-\#(T_j(x)\setminus T_j(x)^{[p]})}{n_j}
 \ge d-\frac1p-\frac1{n_j}.
\end{equation}

Choose integers \(r_j\to\infty\) such that
\(r_j\log(r_j+1)=o(n_j)\).
Partition \(B_j\) according to the vectors
\(
 \bigl(\lfloor r_j a_{j,p}(x)\rfloor\bigr)_{1\le p\le r_j}
 \in\{0,\ldots,r_j\}^{r_j}.
\)
More precisely, for each \(\boldsymbol\alpha=(\alpha_1,\ldots,\alpha_{r_j})
\in\{0,\ldots,r_j\}^{r_j}\), define
\[
 B_{j,\boldsymbol\alpha}
 :=
 \left\{x\in B_j:
 \lfloor r_j a_{j,p}(x)\rfloor=\alpha_p
 \text{ for }1\le p\le r_j\right\}.
\]
Then the nonempty sets \(B_{j,\boldsymbol\alpha}\) form a measurable
partition of \(B_j\) with at most \((r_j+1)^{r_j}\) elements. Take an element
\(B_j'=B_{j,\boldsymbol\alpha_j}\) of maximal \(\lambda\)-measure.
By the assumption on \(\lambda(B_j)\) and the choice of \(r_j\), it follows that
\begin{equation}\label{eq:profile-volume}
 \lim_{j\to\infty}\frac1{n_j}\log\lambda(B_j')=0.
\end{equation}
Moreover,
\begin{equation}\label{uxy}
 |a_{j,p}(x)-a_{j,p}(y)|\le\frac1{r_j}
  \qquad \text{for all } x,y\in B_j' \text{ and } 1\le p\le r_j.
\end{equation}
Fix any \(x_j\in B_j'\). Passing to a diagonal subsequence, we may assume that
\(a_{j,p}(x_j)\to a_p\) for every fixed \(p\ge1\). This together with \eqref{uxy} yields
\begin{equation}\label{ajpuni}
\sup_{x\in B_j'}|a_{j,p}(x)-a_p|\le \sup_{x\in B_j'}|a_{j,p}(x)-a_{j,p}(x_j)|+|a_{j,p}(x_j)-a_p|\longrightarrow 0.
\end{equation}
Passing to the limit in \eqref{tjn}, we obtain
\(a_p\ge d-1/p\).
By the monotonicity of $\{a_{j,p}\}_{p\in \N}$, we know $\{a_p\}_p$ is nondecreasing, thus converges to some $a_{\infty}\in [d,1]$.

\smallskip
By \eqref{ajpuni}, we can choose strictly increasing integers \(j_p\) such that \(
\sup_{x\in B_j'}|a_{j,p}(x)-a_p|\le \frac1p
\) holds for every \(j\ge j_p\).
Define \(P_j:=p\) for \(j_p\le j<j_{p+1}\), ignoring finitely many initial indices.
By the choice of \(j_p\) and the convergence
\(a_{P_j}\to a_\infty\), we have
\begin{equation}\label{fqain}
 \sup_{x\in B_j'}|a_{j,P_j}(x)-a_\infty|
 \le \frac1{P_j}+|a_{P_j}-a_\infty|
 \longrightarrow0.
\end{equation}
For \(x\in B_j'\), put
\(
 \widetilde Q_j(x):=T_j(x)^{[P_j]}.
\)
By \eqref{tpes}, one has that
\begin{equation}\label{fQj}
\partial\widetilde Q_j(x)\subset T_j(x),\qquad
 \#\partial\widetilde Q_j(x)\le 2n_j/P_j+2.
\end{equation}
As \(\widetilde Q_j(x)\subset [0,n_j)_{\mathbb Z}\) is determined by its
boundary, there exist at most
\[
 \sum_{s=0}^{\min\{n_j+1,\lfloor2n_j/P_j+2\rfloor\}}
 \binom{n_j+1}{s}
 =e^{o(n_j)}
\]
choices for $\widetilde Q_j(x)$. Choose a class of maximal \(\lambda\)-measure, let \(Q_j\) be the common set on this class, and set
\[
A_j=\{x\in B_j':\widetilde Q_j(x)=Q_j\}.
\]

\smallskip
Now we show $\{A_j\}$ and $\{Q_j\}$ satisfy our requirements.
Note that
\(
 e^{-o(n_j)}\lambda(B_j')
 \le\lambda(A_j)\le\lambda(B_j'),
\)
which together with \eqref{eq:profile-volume} and \eqref{fqain} yields the assertion \eqref{mAj}. Because $\widetilde Q_j(x)=Q_j$ for every $x\in A_j$, \eqref{fQj} gives assertion \eqref{bQj}. Observe that for every $p\ge 1$, $
T_j(x)^{[p]}\subset Q_j$ for every $x\in A_j$ and all sufficiently large \(j\).
Therefore,
\[
\begin{aligned}
\lim_{p\to\infty}\limsup_{j\to\infty}
\sup_{x\in A_j}\frac{\#(Q_j\setminus T_j(x)^{[p]})}{n_j}
&=
\lim_{p\to\infty}\limsup_{j\to\infty}
\sup_{x\in A_j}\bigl(a_{j,P_j}(x)-a_{j,p}(x)\bigr)\\
&=
\lim_{p\to\infty}(a_\infty-a_p)=0,
\end{aligned}
\]
where we used \eqref{ajpuni} and \eqref{fqain}. This gives assertion \eqref{tpc}.
Assertion \eqref{QcapT} follows from the following observation:
\[
 \begin{aligned}
 \frac{\#(Q_j\cap T_j(x))}{n_j}
 &=\frac{\#T_j(x)-\#(T_j(x)\setminus T_j(x)^{[P_j]})}{n_j}\\
 &\ge d-\frac1{P_j}-\frac1{n_j}.
 \end{aligned}
\]
The proof is now complete.
\end{proof}

Now we can prove Proposition \ref{prop:hyperbolic-folner-selection} by applying Lemma \ref{lem:finite-folner-selection}.

\begin{proof}[Proof of Proposition~\ref{prop:hyperbolic-folner-selection}]
For \(c>0\) and an integer \(L\ge1\), let
\(
 N_0(c,L):=N_0(c/2,c,L,\theta)\ge L
\)
be the integer given by Proposition~\ref{prop:suffix-preball}
with \(\sigma=c/2\).

\smallskip
Consider the countable family of tuples ${\bf{p}}=(c,\eta,\rho,L,N,\ell,M_0)$ satisfying
$$
c,\eta,\rho\in\mathbb Q_{>0},
 \qquad L,\ell,M_0,N\in\mathbb Z_{\ge1},
 \qquad N\ge N_0(c,L).
$$
Define
$$
\Gamma({\bf{p}})=\bigcap_{M\ge M_0}\bigcap_{K\ge 1}\bigcup_{n\ge K}\left\{x\in \Gamma: \#(S_{c,L,N}^E(x)\cap G_{\eta,\ell,M}^{F}(x)
 \cap[0,n))\ge 2 \rho n\right\}.
$$
Since for every $x\in \Gamma$, there exists a $\mu\in\pw(x)$ satisfying \eqref{eq:lvl-joint-signs},
we have $\Gamma=\bigcup_{{\bf{p}}}\Gamma({\bf{p}})$ by applying Theorem~\ref{thm:empirical-synchronization}.
Thus, we can choose a tuple ${\bf{p}}$ such that $\Leb_D(\Gamma({\bf{p}}))>0$ by assumption.

\smallskip
Set $\Gamma_0=\Gamma({\bf{p}})$. As $\Gamma_0\subset D^{\infty}$ by construction, we can apply Proposition~\ref{thm:dynamical-density-new} to $\Gamma_0$, showing that
\[
 \frac{\Leb_D(V_b^D(x)\setminus\Gamma)}
      {\Leb_D(V_b^D(x))}
 \le
 \frac{\Leb_D(V_b^D(x)\setminus\Gamma_0)}
      {\Leb_D(V_b^D(x))}
 \longrightarrow0
 \quad \text{as } b\to\infty,\quad b\in S_{c,L,N}^E(x),
\]
for \(\Leb_D\)-almost every \(x\in\Gamma_0\). By Egorov's theorem, there exists $\Gamma^*\subset\Gamma_0$ of positive $\Leb_D$-measure such that this convergence is uniform, establishing \eqref{eq:uniform-density}.

\smallskip
Choose integers \(M_j\ge M_0\) with \(M_j\to\infty\), and set
\[
 B_{j,n}:=
 \left\{x\in\Gamma^*:\#\left(S_{c,L,N}^E(x)\cap G_{\eta,\ell,M_j}^{F}(x)
 \cap[0,n)\right)\ge2\rho n\right\}.
\]
According to construction, we know \(\Gamma^*\subset\limsup_n B_{j,n}\) for each \(j\).
By the Borel--Cantelli lemma, \(\Leb_D(B_{j,n})\ge e^{-n/j}\) for infinitely many \(n\). Therefore, there is an increasing sequence \(n_j\) such that \(n_j\ge jM_j\) and \(\Leb_D(B_{j,n_j})\ge e^{-n_j/j}\) for every $j\in \N$.

Setting for every $j\in \N$ that
$$
\mathcal T_j(x):=S_{c,L,N}^E(x)\cap G_{\eta,\ell,M_j}^{F}(x)
 \cap[0,n_j),
 $$
$$
 B_j:=B_{j,n_j}=\{x\in \Gamma^*: \#\mathcal T_j(x)\ge 2\rho n_j\}.
$$
Then
$$
\lim_{j\to \infty}\frac1{n_j}\log\Leb_D(B_j)=0.
$$
Now we can apply Lemma~\ref{lem:finite-folner-selection} with $\lambda=\Leb_D$, $T_j=\mathcal T_j$, and $d=2\rho$. This produces measurable sets $E_j\subset B_j$ of positive $\Leb_D$-measure and finite sets $Q_j\subset[0,n_j)_{\mathbb Z}$ satisfying properties \eqref{com1}--\eqref{com3} after discarding finitely many initial indices, since
\[
  \liminf_{j\to\infty}\inf_{x\in E_j} \frac{\#(Q_j\cap\mathcal T_j(x))}{n_j}\ge 2\rho > \rho.
\]
The growth condition \(n_j\ge jM_j\) is preserved under
passage to a subsequence and relabelling, while \eqref{com4} follows
from the choice of \(\Gamma^*\).
\end{proof}

\subsection{Proof of Theorem \ref{thm:lvl-basin-production}}

In this subsection, we shall give the proof of Theorem \ref{thm:lvl-basin-production}.
Fix the parameters and sets given by
Proposition~\ref{prop:hyperbolic-folner-selection}, and define
\begin{equation*}
 m_j:=\frac{\Leb_D|_{E_j}}{\Leb_D(E_j)},
 \qquad
 \nu_j:=\frac1{\#Q_j}\sum_{t\in Q_j}f_*^tm_j.
\end{equation*}
We first establish hyperbolicity and the entropy formula
for weak limits of $\{\nu_j\}$.

\begin{lemma}\label{lem:selected-hyperbolicity}
Every weak limit $\nu$ of $\{\nu_j\}$ is invariant, supported on
$\Lambda$, and satisfies
\begin{equation}\label{eq:selected-signs}
 \lambda_E^-(y)\ge c/2,\qquad
 \lambda_F^+(y)\le-\eta/2
 \quad\text{for }\nu\text{-almost every }y.
\end{equation}
\end{lemma}

\begin{proof}
Passing to a subsequence if necessary, we may assume $\nu_j\weakstar\nu$.
Observe that for any function $\varphi\in C(\mathbf{M})$,
$$
 \left|\int\varphi\circ f\,d\nu_{j}
             -\int\varphi\,d\nu_{j}\right|
 \le\|\varphi\|_\infty\frac{\#\partial Q_j}{\#Q_j}.
$$
Using the facts \(\#Q_j\ge\rho n_j\) and \(\#\partial Q_j=o(n_j)\) provided in Proposition \ref{prop:hyperbolic-folner-selection}, we conclude that
the limit \(\nu\) is \(f\)-invariant. For every fixed \(m\ge1\),
trapping gives
\[
 \nu_j\bigl(\mathbf{M}\setminus f^m(\overline U)\bigr)
 \le\frac{m}{\#Q_j}\longrightarrow0.
\]
Since \(f^m(\overline U)\) is compact, we have
\(\nu(f^m(\overline U))=1\). Hence \(\nu(\Lambda)=1\).

For $R\ge N$, let $H_R$ denote the compact set of points $y\in\overline U$ satisfying:
\begin{itemize}
\item $f^{-s}y\in\overline U$ for all $0\le s\le R$;
\item for all $N\le r \le R$, it holds that
$$
A_r^E(f^{-r}y)\ge \frac{c}{2} r,\qquad \sum_{s=0}^{r-1}\psi_\ell^F(f^{-s}y)\ge\frac\eta2r,
$$
\end{itemize}
where we recall $\psi_\ell^F(w):=-\frac1\ell\log\|Df_w^\ell|_{\widetilde F(w)}\|$.
By definition, the sequence of sets $\{H_R\}_{R\ge N}$ is nested and decreasing. Define $K:=\bigcap_{R\ge N}H_R$. Clearly, $K\subset\bigcap_{R\ge N}f^R(\overline U)=\Lambda$. By Lemma~\ref{lem:cone-bundle-comparison}, the two-sided Oseledets theorem,
Birkhoff's ergodic theorem, and subadditivity along \(F\), we obtain
\[
 \lambda_E^-(y)=\lim_{r\to\infty}\frac1rA_r^E(f^{-r}y)\ge \frac{c}{2},
 \qquad
 \lambda_F^+(y)\le
 -\lim_{r\to\infty}\frac1r
       \sum_{s=0}^{r-1}\psi_\ell^F(f^{-s}y)\le -\frac{\eta}{2}
\]
for $\nu$-almost every $y\in K$. Thus, to establish \eqref{eq:selected-signs} and complete the proof, it suffices to show that
\begin{equation}\label{unionK}
\nu\left(\bigcup_{s\ge 1}f^{-s}K\right)=1.
\end{equation}

For $R\ge N$ and $p\ge 1$, we define
\[
 C_{R,p}:=\bigcup_{s=1}^pf^{-s}H_R,
 \qquad
 e_{j,R}:=\sup_{x\in E_j}\frac1{n_j}
 \#\{t\in\mathcal T_j(x):f^tx\notin H_R\}.
\]
We formulate the following claim, whose proof is deferred to the end.

\begin{claim*}
For all $R\ge N$ and $p\ge 1$,
$$
 \nu_j(C_{R,p}^c)
 \le\frac1\rho\left(
 \sup_{x\in E_j}\frac{\#(Q_j\setminus\mathcal T_j(x)^{[p]})}{n_j}
 +p e_{j,R}\right).
$$
\end{claim*}

Assuming this claim, we proceed to deduce \eqref{unionK} by first estimating $e_{j,R}$. By Lemma~\ref{lem:robust-suffix-conorm}, for every $t\in\mathcal T_j(x)$ with $t\ge R$, we have
\begin{equation}\label{eine}
A_r^E\left(f^{-r}(f^tx)\right)\ge \frac{c}{2} r,\qquad N\le r\le R.
\end{equation}
For the \(F\)-inequalities, let \(\mathcal I_x^{(j)}\) be the union of intervals defining \(G_{\eta,\ell,M_j}^F(x)\), and set $W_j(x):=\mathcal I_x^{(j)}\cap[0,n_j+1)$. By the construction of $\mathcal I_x^{(j)}$, all but at most two connected components of \(W_j(x)\) contain at least \(M_j\) integers. Hence, \(W_j(x)\) has at most $2+(n_j+1)/M_j$ components.

Let \(t\in\mathcal T_j(x)\), and let \(a\) be the first time in the component of \(W_j(x)\) containing \(t\). If \(t\ge a+R\), then the integer $i=t-r+1$ belongs to the same component for all \(N\le r\le R\). Since \(t\in G_{\eta,\ell,M_j}^F(x)\), the defining inequality of this set gives
\[
 \sum_{s=0}^{r-1}\psi_\ell^F(f^{-s}(f^tx))
 =
 \sum_{i=t-r+1}^{t}\psi_\ell^F(f^ix)
 \ge \frac{\eta}{2}r,
 \qquad N\le r\le R.
\]
Combining this with \eqref{eine} for \(t\ge R\), we see that
\[
 \#\{t\in\mathcal T_j(x):f^tx\notin H_R\}
 \le R+R\left(2+\frac{n_j+1}{M_j}\right).
\]
Consequently, for every fixed \(R\ge N\),
\begin{equation*}
 e_{j,R}\le
 \frac{R}{n_j}\left(3+\frac{n_j+1}{M_j}\right)
 \longrightarrow 0 \quad \text{as } j\to\infty,
\end{equation*}
since both \(n_j\to\infty\) and \(M_j\to\infty\). Since \(C_{R,p}\) is compact, it follows from the claim that
\[
 \nu(C_{R,p})
 \ge\limsup_{j\to\infty}\nu_j(C_{R,p})
 \ge
 1-\frac1\rho\limsup_{j\to\infty}
 \sup_{x\in E_j}
 \frac{\#(Q_j\setminus\mathcal T_j(x)^{[p]})}{n_j}.
\]
Observing that $\bigcap_{R\ge N}C_{R,p}=\bigcup_{s=1}^pf^{-s}K$ and invoking \eqref{eq:prop-filling}, we conclude
\[
\begin{aligned}
\nu\left(\bigcup_{s\ge 1}f^{-s}K\right)
&=\lim_{p\to \infty}\nu\left(\bigcap_{R\ge N}C_{R,p}\right)
 =\lim_{p\to \infty}\lim_{R\to \infty}\nu(C_{R,p})\\
&\ge 1-\frac{1}{\rho}\lim_{p\to \infty}\limsup_{j\to \infty} \sup_{x\in E_j}\frac{\#(Q_j\setminus \mathcal T_j(x)^{[p]})}{n_j}\\
&=1,
\end{aligned}
\]
which proves \eqref{unionK}.

\medskip

\noindent\emph{Proof of the claim.}
Fix \(x\in E_j\). Suppose that
\(t\in\mathcal T_j(x)^{[p]}\) and \(f^tx\notin C_{R,p}\).
By the definition of the filling, there exist
\(a,u\in\mathcal T_j(x)\) such that
\(a\le t<u\) and \(u-a\le p\). Hence \(1\le u-t\le p\).
If \(f^ux\in H_R\), then
\(
 f^tx\in f^{-(u-t)}H_R\subset C_{R,p},
\)
a contradiction. Thus \(f^ux\notin H_R\), and consequently
\[
 \bigl\{t\in\mathcal T_j(x)^{[p]}:
              f^tx\notin C_{R,p}\bigr\}
 \subset
 \bigcup_{\substack{u\in\mathcal T_j(x)\\
                    f^ux\notin H_R}}
 [u-p,u)_{\mathbb Z}.
\]
Therefore,
\[
 \begin{aligned}
 \#\{t\in Q_j:f^tx\notin C_{R,p}\}
 &\le
 \#\bigl(Q_j\setminus\mathcal T_j(x)^{[p]}\bigr)
 +p\,\#\{u\in\mathcal T_j(x):f^ux\notin H_R\}\\
 &\le
 \#\bigl(Q_j\setminus\mathcal T_j(x)^{[p]}\bigr)
 +pn_j e_{j,R}.
 \end{aligned}
\]
Integrating with respect to \(m_j\) and using
\(\#Q_j\ge\rho n_j\), we conclude that
\[
 \begin{aligned}
 \nu_j(C_{R,p}^{c})
 &=\frac1{\#Q_j}
   \int_{E_j}
   \#\{t\in Q_j:f^tx\notin C_{R,p}\}\,dm_j(x)\\
 &\le\frac1\rho
 \left(
   \sup_{x\in E_j}
   \frac{\#\bigl(Q_j\setminus\mathcal T_j(x)^{[p]}\bigr)}{n_j}
   +p e_{j,R}
 \right),
 \end{aligned}
\]
as claimed.
\end{proof}

\begin{lemma}\label{lem:selected-entropy}
Every weak limit $\nu$ of $\{\nu_j\}$ is a hyperbolic SRB measure satisfying the Pesin entropy formula:
$$
 h_\nu(f)=\int_{\Lambda}\sum_{\lambda_i(x)>0}\lambda_i(x)\,d\nu=\int_\Lambda\log\Jac(Df|_E)\,d\nu.
$$
\end{lemma}

\begin{proof}
Without loss of generality, we assume $\nu_j\weakstar\nu$. By Lemma \ref{lem:selected-hyperbolicity}, we know $\nu$ is a hyperbolic measure supported on $\Lambda$, which admits positive Lyapunov exponents along $E$, and negative Lyapunov exponents along $F$. Thus, by Ruelle's inequality \cite{Ruelle} we get
$$
h_\nu(f)\le \int_{\Lambda}\sum_{\lambda_i(x)>0}\lambda_i(x)\,d\nu=\int_\Lambda\log\Jac(Df|_E)\,d\nu.
$$
By the result of Ledrappier--Young \cite[Theorem A]{LedrappierYoung}, we need only to show
$$
h_\nu(f)\ge \int_\Lambda\log\Jac(Df|_E)\,d\nu.
$$

Fix $\varepsilon>0$ and a finite Borel partition $\cP$ with
$\nu(\partial\cP)=0$ and sufficiently small diameter for
Theorem~\ref{thm:multi-gibbs}. Every point of $E_j$ satisfies the
boundary-distance and suffix-time conditions of that theorem. Thus,
for $x\in E_j$,
\[
 m_j(\cP^{Q_j}(x))
 \le\frac{C^{\#\partial Q_j}
             e^{\varepsilon\#Q_j-J_{Q_j}(x)}}{\Leb_D(E_j)}.
\]
Integrating the negative logarithm gives
\begin{equation}\label{eq:selected-partition-entropy}
 H_{m_j}(\cP^{Q_j})
 \ge\int J_{Q_j}\,dm_j-\varepsilon\#Q_j
       -\#\partial Q_j\log C+\log\Leb_D(E_j).
\end{equation}
Combining \eqref{eq:selected-partition-entropy} with the
entropy estimate in \cite[proof of Lemma~5]{BurguetSurface},
we obtain, for every integer \(r\ge1\),
\begin{align}
 \frac1rH_{\nu_j}(\cP^r)
 \ge{}&
 \frac1{\#Q_j}\int J_{Q_j}\,dm_j-\varepsilon
 +\frac{\log\Leb_D(E_j)}{\#Q_j}\notag\\
 &-\bigl(\log C+3r\log\#\cP\bigr)
   \frac{\#\partial Q_j}{\#Q_j}.
 \label{eq:selected-entropy-comparison}
\end{align}

Now we show
\begin{equation}\label{eq:selected-Jacobian-limit}
 \lim_{j\to\infty}\frac1{\#Q_j}\int J_{Q_j}\,dm_j
 =\int_\Lambda\log\Jac(Df|_E)\,d\nu.
\end{equation}
To this end, consider the lifted measures
\[
 \widehat\nu_j:=\frac1{\#Q_j}\int_{E_j}
       \sum_{t\in Q_j}\delta_{(f^tx,T_{f^tx}f^tD)}\,dm_j(x),
 \qquad
 \widehat f(y,P):=(fy,Df_yP).
\]
The boundary estimate shows that every weak limit $\widehat\nu$ is
$\widehat f$-invariant and projects to $\nu$.
By cone invariance and Lemma~\ref{lem:selected-hyperbolicity},
it is supported on cone planes over $\Lambda$. Domination and invariance of $E$ on $\Lambda$
give constants $A\ge1$ and $\tau\in(0,1)$ such that
\[
 d_{\Gr}(Df_y^tP,E(f^ty))\le A\tau^t,
 \quad
y\in\Lambda,\ P\subset\cC_\theta^E(y),\ \dim P=\dim E.
\]
Consequently,
\[
 \int d_{\Gr}(P,E(y))\,d\widehat\nu(y,P)
 =\int d_{\Gr}(Df_y^tP,E(f^ty))\,d\widehat\nu(y,P)
 \le A\tau^t\longrightarrow0.
\]
Thus every such limit is supported on the section $P=E(y)$.
Continuity of $(y,P)\mapsto\log\Jac(Df_y|_P)$ now yields
\eqref{eq:selected-Jacobian-limit}.

For fixed $r$, the last two terms in
\eqref{eq:selected-entropy-comparison} tend to zero because
\[
 \#Q_j\ge\rho n_j,\qquad
 \#\partial Q_j=o(n_j),\qquad
 \log\Leb_D(E_j)=o(n_j),
\]
by applying Proposition~\ref{prop:hyperbolic-folner-selection}.
By invariance of $\nu$ and $\nu(\partial\cP)=0$, one has that
$H_{\nu_j}(\cP^r)\to H_\nu(\cP^r)$.
Letting $j\to\infty$ and then $r\to\infty$ in
\eqref{eq:selected-entropy-comparison}, we obtain
\[
 h_\nu(f)\ge h_\nu(f,\cP)
 \ge\int_\Lambda\log\Jac(Df|_E)\,d\nu-\varepsilon.
\]
Finally let $\varepsilon\downarrow0$.
\end{proof}

We close this subsection by proving Theorem \ref{thm:lvl-basin-production}.

\begin{proof}[Proof of Theorem~\ref{thm:lvl-basin-production}]
Define
\[
 \beta_j:=\frac1{\#Q_j}\int_{E_j}
 \sum_{t\in Q_j\cap\mathcal T_j(x)}\delta_{f^tx}\,dm_j(x).
\]
By Proposition~\ref{prop:hyperbolic-folner-selection},
\[
 0\le\beta_j\le\nu_j,
 \qquad
 \beta_j(\mathbf{M})\ge\frac{\rho n_j}{\#Q_j}\ge\rho.
\]
Passing to a joint subsequence, assume that
\(\nu_j\weakstar\nu\) and \(\beta_j\weakstar\beta\).
Then \(0<\beta\le\nu\). By
Lemmas~\ref{lem:selected-hyperbolicity} and~\ref{lem:selected-entropy},
\(\nu\) is a hyperbolic SRB measure supported on \(\Lambda\).
Almost every ergodic component of \(\nu\) is a hyperbolic SRB measure.
Since there are
at most countably many such measures, we can write
\[
 \nu=\sum_iw_i\mu_i,\qquad
 \sum_iw_i=1,\qquad
 \lambda_E^-(\mu_i)\ge c/2,\quad
 \lambda_F^+(\mu_i)\le-\eta/2.
\]
The basins \(\Basin(\mu_i)\) are pairwise disjoint and have full
\(\nu\)-measure in their union. Hence some \(\mu_*=\mu_i\) satisfies
\[
 0<\beta_*:=\beta|_{\Basin(\mu_*)}
 \le\nu|_{\Basin(\mu_*)}=w_i\mu_*.
\]

\smallskip
Choose a compact Pesin block \(K_*\) with \(\beta_*(K_*)>0\), and let
\(V\) be a sufficiently small neighborhood of a point in
\(\supp(\beta_*|_{K_*})\). Then
\(\beta_*(K_*\cap V)>0\) and, since \(\beta_*\ll\mu_*\),
\(\mu_*(K_*\cap V)>0\).
The SRB property and absolute continuity of the stable lamination
on \(K_*\), with uniform Jacobian bounds, give a constant \(c_*>0\)
such that
\begin{equation}\label{eq:lvl-terminal-basin-fraction}
 \frac{\Leb_\Sigma(\Sigma\cap\Basin(\mu_*))}{\Leb_\Sigma(\Sigma)}
 \ge c_*
\end{equation}
for every terminal disk \(\Sigma=B_{f^tD}(w,R_N)\) from the preball
construction with \(w\in V\); see \cite{Pesin,PesinSinai}.
Here \(V\) is chosen small relative to \(R_N\) and the local stable
manifold size, and the bound is uniform by
Lemma~\ref{tangent-control} and \eqref{eq:local-graph-volume}.

\smallskip
Let \(\mathscr L=\mathscr L(c,D)\ge1\) be the constant from
Corollary~\ref{prop:bounded-distortion}.
By \eqref{eq:uniform-density}, choose \(T\ge b_*(r_*,N)\) such that
\begin{equation}\label{eq:lvl-density-threshold}
 \frac{\Leb_D(V_t^D(x)\setminus\Gamma)}{\Leb_D(V_t^D(x))}
 <\frac{c_*}{2\mathscr L^2}
 \qquad \text{for every }
x\in\Gamma^*,\ t\in S_{c,L,N}^E(x),\ t\ge T.
\end{equation}
Removing the terms with \(t<T\) from \(\beta_j\) changes its mass
by at most \(T/\#Q_j\to0\). Since \(V\) is open,
it follows that
\[
 \liminf_{j\to\infty}\frac1{\#Q_j}\int_{E_j}
 \#\{t\in Q_j\cap\mathcal T_j(x):t\ge T,\ f^tx\in V\}
 \,dm_j(x)
 \ge\beta(V)\ge\beta_*(V)>0.
\]
Thus there exist \(j\), \(x\in E_j\subset\Gamma^*\), and
\(t\in Q_j\cap\mathcal T_j(x)\) with \(t\ge T\) and \(f^tx\in V\).
Set
\[
 P:=V_t^D(x),\qquad
 \Sigma:=f^tP=B_{f^tD}(f^tx,R_N).
\]
Since \(t\in S_{c,L,N}^E(x)\) and \(R_N<\rho_N\), bounded
distortion and change of variables give
\[
 \frac{\Leb_D(P\cap\Basin(\mu_*))}{\Leb_D(P)}
 \ge\frac1{\mathscr L^2}
      \frac{\Leb_\Sigma(\Sigma\cap\Basin(\mu_*))}{\Leb_\Sigma(\Sigma)}
 \ge\frac{c_*}{\mathscr L^2},
\]
where we used basin invariance and
\eqref{eq:lvl-terminal-basin-fraction}.
Combining this with \eqref{eq:lvl-density-threshold}, we obtain
\[
 \begin{aligned}
 \frac{\Leb_D(P\cap\Gamma\cap\Basin(\mu_*))}{\Leb_D(P)}
 &\ge
 \frac{\Leb_D(P\cap\Basin(\mu_*))}{\Leb_D(P)}
 -\frac{\Leb_D(P\setminus\Gamma)}{\Leb_D(P)}\\
 &>\frac{c_*}{2\mathscr L^2}>0.
 \end{aligned}
\]
Thus \(\Leb_D(\Gamma\cap\Basin(\mu_*))>0\), as required.
\end{proof}

\section{Examples and robustness}
\label{sec:examples}

In this section, we first construct a connected attracting set with two \(u\)-definite
center subbundles of dimension two and prove that \(u\)-definiteness
persists under perturbations. Then we discuss physical measures with
distinct unstable indices in Remark~\ref{rem:distinct-indices}.

\subsection{A two-center skew product}

\begin{proposition}
\label{thm:multiple-center-subbundles-example}
There exist a \(C^\infty\) diffeomorphism \(F\) of
\(\mathbf{M}=\mathbb T^2\times\mathbb S^2\times\mathbb S^2\) and a connected
attracting set \(\Lambda\), with trapping region \(U\), admitting a
partially hyperbolic splitting
$$
 T_\Lambda \mathbf{M}=E^u\oplus_{\succ} C_1\oplus_{\succ} C_2\oplus_{\succ} E^s,
 \qquad \dim C_1=\dim C_2=2.
$$
Each \(C_i\) is \(u\)-definite, is neither uniformly expanding nor
uniformly contracting, and admits no nontrivial dominated refinement.
There are exactly two ergodic Gibbs \(u\)-states \(\mu_-,\mu_+\): all
center exponents of \(\mu_-\) are negative, whereas those of \(\mu_+\)
are positive along \(C_1\) and negative along \(C_2\). Thus
\(C_1\oplus C_2\) is not \(u\)-definite. Moreover, \(\mu_-\) is the
unique physical measure supported on \(\Lambda\), and
\[
 \Leb\bigl(B(\Lambda)\setminus\Basin(\mu_-)\bigr)=0.
\]
\end{proposition}

\begin{proof}
We add a second fiber to the Anosov--north--south product in
\cite{CaoMiYang}, choosing the fiber maps to ensure domination between
the center subbundles.

\smallskip
\noindent\emph{The fiber maps.}
Let \(A\colon\mathbb T^2\to\mathbb T^2\) be a mixing linear Anosov
automorphism with fixed point \(x_*=0\), and let \(m_A\) be normalized
Haar measure. Fix \(\theta/\pi\notin\mathbb Q\). Identify the first
sphere with \(\mathbb C\cup\{\infty\}\) and set
\[
 h(w)=8e^{i\theta}w,
 \qquad p_+=0,\quad p_-=\infty.
\]
Then \(p_+\) is a source and \(p_-\) is a sink. In the round metric,
\begin{equation}\label{eq:example-north-south}
 m(Dh_w)=\|Dh_w\|
 =\frac{8(1+|w|^2)}{1+64|w|^2},
 \qquad \frac18\le m(Dh_w)\le8.
\end{equation}
Choose smooth functions \(\chi\colon\mathbb T^2\to[0,1]\) and
\(\rho\colon\mathbb S^2\to[0,1]\) such that
\[
 \chi(x_*)=1,\qquad \int\chi\,dm_A<\frac1{10},
 \qquad
 \rho=1\text{ near }p_+,
 \qquad
 \supp\rho\subset\{y:m(Dh_y)>4\}.
\]
In particular, \(\rho=0\) near \(p_-\).

On the second sphere, use a complex coordinate \(z\) and let
\(D=\{|z|<1\}\). Choose a metric that is Euclidean near \(\overline D\)
and a smooth nonincreasing function
\(\omega\colon[0,\infty)\to[0,1]\), equal to one near zero and to
zero on \([1/2,\infty)\). For each \((x,y)\), let \(g_{x,y}\) be the
time-one map of the vector field
\begin{equation*}
 X_{x,y}(z)
 =\bigl[-\log16+\log32\,\chi(x)\rho(y)\omega(|z|^2)\bigr]z
   +i\theta z.
\end{equation*}
Outside a compact subset of the chart, this vector field equals
\((-\log16+i\theta)z\), so it extends smoothly across infinity.
Thus the maps \(g_{x,y}\) form a smooth family of sphere
diffeomorphisms fixing zero. The radial component points strictly inward near
\(\partial D\), giving \(g_{x,y}(\overline D)\subset D\).

Since \(\omega'\le0\), for every real tangent vector \(v\),
\[
 \langle D_zX_{x,y}v,v\rangle
 \le\bigl[-\log16+\log32\,\chi(x)\rho(y)\bigr]\|v\|^2.
\]
As \(\overline D\) is forward invariant under the flow, the
variational equation gives
\begin{equation}\label{eq:example-fiber-bound}
 \|D_zg_{x,y}\|
 \le \exp\bigl(-\log16+\log32\,\chi(x)\rho(y)\bigr)
 \le2
 \qquad(z\in\overline D).
\end{equation}
If \(\rho(y)=0\), then \(g_{x,y}|_D=16^{-1}R_\theta\), where
\(R_\theta\) denotes rotation through angle \(\theta\). Hence
\begin{equation}\label{eq:example-center-gap}
 \frac{\|D_zg_{x,y}\|}{m(Dh_y)}\le\frac12
 \qquad(x\in\mathbb T^2,\ y\in\mathbb S^2,\ z\in\overline D).
\end{equation}
Indeed, on \(\supp\rho\) the numerator is at most two and the
denominator exceeds four; outside \(\supp\rho\), the bound follows
from \eqref{eq:example-north-south} and \(\|D_zg_{x,y}\|=1/16\).

\smallskip
\noindent\emph{The attracting set and the splitting.}
For a sufficiently large integer \(N\), define
\begin{equation*}
 F(x,y,z)=\bigl(A^Nx,h(y),g_{x,y}(z)\bigr)
 \quad\text{on}\quad
 \mathbf{M}=\mathbb T^2\times\mathbb S^2\times\mathbb S^2.
\end{equation*}
This is a smooth diffeomorphism. Set
\[
 U=\mathbb T^2\times\mathbb S^2\times D,
 \qquad
 \Lambda=\bigcap_{n\ge0}F^n(\overline U).
\]
Then \(F(\overline U)\subset U\), and the sets
\(F^n(\overline U)\) are compact, connected and nested. Thus
\(\Lambda\) is a connected attracting set. Since each \(g_{x,y}\)
fixes zero,
\[
 \mathbb T^2\times\mathbb S^2\times\{0\}\subset\Lambda.
\]

Choose \(N\) so that the Anosov rates dominate the derivatives and
inverse derivatives of the fiber maps. The graph transform gives a
global partially hyperbolic splitting with one-dimensional strong
subbundles and center tangent to the sphere factors; see \cite{CP}.
On \(\Lambda\), let \(C_2\) be tangent to the second sphere. The
center derivative has the triangular form
\[
 \begin{pmatrix}
 Dh_y&0\\
 D_yg_{x,y}(z)&D_zg_{x,y}
 \end{pmatrix}.
\]
By \eqref{eq:example-center-gap}, the graph transform on complements
of \(C_2\) contracts by a factor at most \(1/2\). Its invariant
graph defines \(C_1\), and projection onto the first sphere
conjugates \(DF|_{C_1}\) to \(Dh\). Iterating
\eqref{eq:example-center-gap} and using the uniformly bounded graph
identifications gives the dominated splitting 
$$
 E^u\oplus_{\succ} C_1\oplus_{\succ} C_2\oplus_{\succ} E^s,
\qquad (\dim E^u,\dim C_1,\dim C_2,\dim E^s)=(1,2,2,1).
$$

\smallskip
\noindent\emph{Gibbs \(u\)-states and their Lyapunov exponents.}
Let \(\nu\) be a Gibbs \(u\)-state supported on \(\Lambda\).
Projection \(\pi_x\) onto the Anosov factor maps strong unstable
plaques locally diffeomorphically to unstable plaques of \(A^N\).
Pushing forward the absolutely continuous conditional measures and
disintegrating over base plaques shows that \((\pi_x)_*\nu\) is a
Gibbs \(u\)-state of \(A^N\). By uniqueness for a mixing Anosov
automorphism \cite{Bowen75},
\[
 (\pi_x)_*\nu=m_A.
\]
Set
\begin{equation*}
 \kappa=\log16-\log32\int\chi\,dm_A>0.
\end{equation*}
Then \eqref{eq:example-fiber-bound} gives
\begin{equation}\label{eq:example-negative-gibbs-average}
 \int\log\|DF|_{C_2}\|\,d\nu
 \le-\log16+\log32\int\chi\,dm_A=-\kappa.
\end{equation}

In fact, every invariant measure \(\nu\) on \(\Lambda\) projecting
to \(m_A\) is supported on \(z=0\). Write
\(F^n(x,y,z)=(A^{Nn}x,h^ny,z_n)\). Since \(g_{x,y}(0)=0\),
\eqref{eq:example-fiber-bound} implies
\[
 |z_n|
 \le |z|\exp\left(
   -n\log16+\log32\sum_{j=0}^{n-1}\chi(A^{Nj}x)
 \right).
\]
Birkhoff's ergodic theorem and \((\pi_x)_*\nu=m_A\) give \(z_n\to0\)
for \(\nu\)-almost every point. Invariance and dominated convergence
then yield \(\int|z|^2\,d\nu=\int|z_n|^2\,d\nu\to0\).

Since \(p_-\) and \(p_+\) are the only recurrent points of \(h\),
the only possible ergodic Gibbs \(u\)-states on \(\Lambda\) are
\begin{equation}\label{eq:example-two-gibbs-states}
 \mu_-=m_A\times\delta_{p_-}\times\delta_0,
 \qquad
 \mu_+=m_A\times\delta_{p_+}\times\delta_0.
\end{equation}
The strong unstable plaques in each supporting torus are those of
the Anosov factor, so both measures are Gibbs \(u\)-states.
At \(z=0\),
\[
 D_zg_{x,y}(0)
 =\exp\bigl(-\log16+\log32\,\chi(x)\rho(y)\bigr)R_\theta.
\]
Since \(Dh\) and \(D_zg_{x,y}(0)\) are conformal, the two
exponents along each center subbundle coincide. Their values are
\begin{equation*}
 \begin{array}{c|cc}
       &C_1&C_2\\ \hline
 \mu_- &-\log8&-\log16\\
 \mu_+ & \log8&-\kappa
 \end{array}.
\end{equation*}
Thus both center subbundles are \(u\)-definite, whereas
\(C_1\oplus C_2\) is not.

\smallskip
\noindent\emph{Nonuniformity and absence of dominated refinements.}
The fixed points \(q_\pm=(x_*,p_\pm,0)\) belong to \(\Lambda\),
and their center derivatives are
\begin{equation*}
 \begin{aligned}
 DF_{q_-}|_{C_1}&=8^{-1}R_{-\theta},
 &DF_{q_-}|_{C_2}&=16^{-1}R_\theta,\\
 DF_{q_+}|_{C_1}&=8R_\theta,
 &DF_{q_+}|_{C_2}&=2R_\theta.
 \end{aligned}
\end{equation*}
These fixed points have unstable indices one and five, respectively.
Each center subbundle contracts at \(q_-\) and expands at \(q_+\),
so neither is uniformly expanding or uniformly contracting. The two
center restrictions at \(q_-\) have no real invariant line. Since
each \(C_i\) has dimension two, it admits no nontrivial dominated
refinement on \(\Lambda\).

\smallskip
\noindent\emph{The physical measure.}
For \(y\ne p_+\), we have \(h^ny\to p_-\). Eventually
\(\rho(h^ny)=0\), after which the second fiber contracts by
\(1/16\) at each iterate. Hence, for every
\(x\in\Basin_{A^N}(m_A)\), every \(y\ne p_+\), and every \(z\in D\),
\[
 \frac1n\sum_{j=0}^{n-1}\delta_{F^j(x,y,z)}\weakstar\mu_-.
\]
Thus \(\Basin(\mu_-)\) has full Lebesgue measure in \(U\), and
hence in \(B(\Lambda)\), by pulling back the exceptional set along
entrance times.

Let \(\nu\) be any physical measure supported on \(\Lambda\).
A continuous function equal to one on \(\Lambda\) and supported in
\(U\) shows that every point of \(\Basin(\nu)\) eventually enters
\(U\). Since \(\Basin(\nu)\) is invariant and has positive volume,
its intersection with \(U\) has positive volume. It therefore meets
\(\Basin(\mu_-)\), so \(\nu=\mu_-\). This proves uniqueness and
completes the proof.
\end{proof}

\subsection{Robustness}

The \(u\)-definiteness in
Proposition~\ref{thm:multiple-center-subbundles-example} persists under
small \(C^1\) perturbations within the \(C^{1+\alpha}\) class,
without requiring the skew-product structure to be preserved.

\begin{proposition}\label{prop:robust-two-center}
The map \(F\) has a \(C^1\)-neighborhood \(\cO\) such that,
for every \(g\in\cO\), the region \(U\) remains trapping and its
attracting continuation \(\Lambda_g\) admits the continued partially
hyperbolic splitting
\[
 T_{\Lambda_g}\mathbf{M}=E_g^u\oplus_{\succ} C_{1,g}\oplus_{\succ} C_{2,g}\oplus_{\succ} E_g^s.
\]
For \(g\in\cO\cap\Diff^{1+\alpha}(\mathbf{M})\), both center subbundles
are \(u\)-definite. Neither is uniformly expanding or uniformly
contracting, and neither admits a nontrivial dominated refinement.
\end{proposition}

\begin{proof}
Choose a \(C^1\)-neighborhood \(\cO\) of \(F\) such that every
\(g\in\cO\) remains globally partially hyperbolic,
\(g(\overline U)\subset U\), and
\[
 \Lambda_g=\bigcap_{n\ge0}g^n(\overline U)
\]
admits the continued dominated splitting. 
We first consider \(C_{2,g}\). By the \(C^1\) upper
semicontinuity of Gibbs \(u\)-states
\cite[Theorem~B]{YangEntropy}, if
\(g_j\in\cO\cap\Diff^{1+\alpha}(\mathbf{M})\) converges to \(F\)
and \(\nu_j\) is a Gibbs \(u\)-state of \(g_j\) supported on
\(\Lambda_{g_j}\), every weak-* limit is a Gibbs \(u\)-state of
\(F\) supported on \(\Lambda\).
Together with continuity of the continued subbundles and
\eqref{eq:example-negative-gibbs-average}, this gives, after
shrinking \(\cO\),
\[
 \int\log\|Dg|_{C_{2,g}}\|\,d\nu\le-\frac{\kappa}{2}
\]
for every \(g\in\cO\cap\Diff^{1+\alpha}(\mathbf{M})\) and every
Gibbs \(u\)-state \(\nu\) of \(g\) supported on \(\Lambda_g\).
For ergodic \(\nu\), subadditivity gives
\[
 \lambda_{C_{2,g}}^+(\nu)
 \le \int\log\|Dg|_{C_{2,g}}\|\,d\nu
 \le-\frac{\kappa}{2}.
\]
Thus \(C_{2,g}\) is \(u\)-definite.

For \(C_{1,g}\), we use the source--sink structure of \(h\).
Since \(\rho\) is constant near \(p_\pm\), we have
\(D_yg_{x,y}=0\) on the invariant slices \(y=p_\pm\).
Uniqueness of the invariant graph defining \(C_1\) therefore
implies that \(C_1\) is tangent to the first sphere on these
slices. Its derivative there is \(8R_\theta\) and
\(8^{-1}R_{-\theta}\), respectively.

Choose sufficiently small spherical disks \(V_\pm\) about
\(p_\pm\), with disjoint closures, and put
\(W_\pm=\pi_y^{-1}(V_\pm)\). By continuity of the subbundles
and upper semicontinuity of the attracting continuation,
after shrinking \(\cO\) we have, for every \(g\in\cO\),
\[
 \begin{aligned}
 m(Dg|_{C_{1,g}})&>2
 &&\text{on }\Lambda_g\cap W_+,\\
 \|Dg|_{C_{1,g}}\|&<\frac12
 &&\text{on }\Lambda_g\cap W_-.
 \end{aligned}
\]
We now keep \(V_\pm\) fixed. Since \(p_+\) is a source and
\(p_-\) is a sink, a further shrinking of \(\cO\) ensures
\[
 g^{-1}(\overline W_+)\subset W_+,
 \qquad
 g(\overline W_-)\subset W_-
 \qquad(g\in\cO).
\]
Moreover, the function
\[
 H(y)=\frac{|y|^2}{1+|y|^2},
 \qquad H(p_-)=1,
\]
satisfies \(H(h(y))>H(y)\) away from \(p_\pm\).
By compactness of the annulus outside \(V_+\cup V_-\),
after shrinking \(\cO\) there exists \(\delta>0\) such that
\[
 H(\pi_y(gw))-H(\pi_y(w))\ge\delta
 \qquad
 \text{for every }g\in\cO
 \text{ and }w\notin W_+\cup W_-.
\]
Thus an orbit outside \(W_+\) cannot enter \(W_+\) and
reaches \(W_-\) in uniformly bounded time. Poincar\'e
recurrence and ergodicity imply that every ergodic invariant
measure on \(\Lambda_g\) gives full measure to either \(W_+\)
or \(W_-\). The derivative bounds above then give only
positive or only negative exponents along \(C_{1,g}\).
In particular, \(C_{1,g}\) is \(u\)-definite.

Finally, after shrinking \(\cO\), the hyperbolic fixed points
\(q_\pm\) have continuations \(q_\pm(g)\in U\). Since these
points are fixed, they belong to \(\Lambda_g\). For each
\(i=1,2\), the eigenvalues of \(Dg|_{C_{i,g}}\) have modulus
less than one at \(q_-(g)\) and greater than one at \(q_+(g)\).
Hence neither center subbundle is uniformly expanding or
uniformly contracting. At \(q_-(g)\), both center restrictions
retain nonreal eigenvalues and therefore admit no real
invariant line. Since each \(C_{i,g}\) is two-dimensional,
neither admits a nontrivial dominated refinement.
\end{proof}

For every \(g\in\cO\cap\Diff^{1+\alpha}(\mathbf{M})\),
Proposition~\ref{prop:robust-two-center} and Theorem~\ref{thm:A}
give finitely many ergodic hyperbolic SRB measures whose basins
cover \(B(\Lambda_g)\) modulo zero.
The classification \eqref{eq:example-two-gibbs-states} and the failure
of \(u\)-definiteness for \(C_1\oplus C_2\) are asserted only for
\(F\).

\begin{remark}
\label{rem:distinct-indices}
The physical measures in Theorem~\ref{thm:A} need not have the same
unstable index. Indeed, one can construct three pairwise disjoint
mixing hyperbolic attractors on a connected manifold with center sign
patterns
\[
 (-,-),\qquad (+,-),\qquad (+,+),
\]
respectively. They can be realized simultaneously by a skew-product
construction,\footnote{For example, take a Morse--Smale circle map
with three attracting fixed points and choose the two center fiber
maps near these points with signs \((-,-)\), \((+,-)\), and \((+,+)\),
respectively.}
with a partially hyperbolic splitting
\(E^u\oplus_{\succ} C_1\oplus_{\succ} C_2\oplus_{\succ} E^s\), where
\(\dim C_1=\dim C_2=2\). Both center subbundles are therefore
\(u\)-definite.

Each attractor carries a unique physical SRB measure, with unstable
indices \(1,3,5\), respectively. Their basins cover the basin of
the union modulo zero. These properties persist in a
\(C^1\)-neighborhood by uniform hyperbolicity.
\end{remark}

\end{sloppypar}
\end{document}